\documentclass[12pt]{amsart}
\usepackage{amssymb,amsmath,amsfonts,latexsym, mathtools}
\usepackage{bm, enumerate}
\usepackage{graphicx}
\usepackage{color,soul}
\usepackage{hyperref}
\usepackage{dsfont}
\newcommand{\newsection}[1]{\setcounter{equation}{0} \section{#1}}
\newcommand{\bea}{\begin{eqnarray}}
\newcommand{\eea}{\end{eqnarray}}

\newcommand{\clb}{\mathcal{B}}

\newcommand{\cld}{\mathcal{D}}
\newcommand{\cle}{\mathcal{E}}

\newcommand{\clh}{\mathcal{H}}
\newcommand{\clk}{\mathcal{K}}
\newcommand{\cll}{\mathcal{L}}
\newcommand{\clm}{\mathcal{M}}

\newcommand{\clq}{\mathcal{Q}}

\newcommand{\cls}{\mathcal{S}}

\newcommand{\clw}{\mathcal{W}}

\newcommand{\br}{\mathbf r}

\newcommand{\D}{\mathbb{D}}

\newcommand{\C}{\mathbb{C}}
\newcommand{\Z}{\mathbb{Z}}
\newcommand{\z}{\bm{z}}
\newcommand{\w}{\bm{w}}

\newcommand{\raro}{\rightarrow}

\def \qed {\hfill \vrule height6pt width 6pt depth 0pt}
\def\textmatrix#1&#2\\#3&#4\\{\bigl({#1 \atop #3}\ {#2 \atop #4}\bigr)}
\def\dispmatrix#1&#2\\#3&#4\\{\left({#1 \atop #3}\ {#2 \atop #4}\right)}
\newcommand{\be}{\begin{equation}}
\newcommand{\ee}{\end{equation}}
\newcommand{\ben}{\begin{eqnarray*}}
\newcommand{\een}{\end{eqnarray*}}

\newcommand{\bi}{\begin{itemize}}
\newcommand{\ei}{\end{itemize}}

\newtheorem{Theorem}{\sc Theorem}[section]
\newtheorem{Lemma}[Theorem]{\sc Lemma}
\newtheorem{Proposition}[Theorem]{\sc Proposition}
\newtheorem{Corollary}[Theorem]{\sc Corollary}
\newtheorem{Definition}[Theorem]{\sc Definition}
\newtheorem{Example}[Theorem]{\sc Example}
\newtheorem{Examples}[Theorem]{\sc Examples}
\newtheorem{Remark}[Theorem]{\sc Remark}

\newtheorem{Note}[Theorem]{\sc Note}
\newtheorem{Question}{\sc Question}
\newtheorem{ass}[Theorem]{\sc Assumption}
\newcommand{\bt}{\begin{Theorem}}
\def\beginlem{\begin{Lemma}}
\def\beginprop{\begin{Proposition}}
\def\begincor{\begin{Corollary}}
\def\begindef{\begin{Definition}}
\def\beginexamp{\begin{Example}}
\def\beginrem{\begin{Remark}}
\def\beginq{\begin{Question}}
\def\beginass{\begin{ass}}
\def\beginnote{\begin{Note}}
\newcommand{\et}{\end{Theorem}}
\def\endlem{\end{Lemma}}
\def\endprop{\end{Proposition}}
\def\endcor{\end{Corollary}}
\def\enddef{\end{Definition}}
\def\endexamp{\end{Example}}
\def\endrem{\end{Remark}}
\def\endq{\end{Question}}
\def\endass{\end{ass}}
\def\endnote{\end{Note}}

\begin{document}

\title{$\clw$-hypercontractions and their model}

\dedicatory{To the memory of Professor J\"{o}rg Eschmeier}

\author[Bhattacharjee]{Monojit Bhattacharjee}
\address{Department of Mathematics, Birla Institute of Technology and Science - Pilani, K. K. Birla Goa Campus, South Goa, 403726, India}
\email{monojitb@goa.bits-pilani.ac.in, monojit.hcu@gmail.com}

\author[Das]{B. Krishna Das}
\address{Department of Mathematics, Indian Institute of Technology Bombay, Powai, Mumbai, 400076, India}
\email{bata436@gmail.com, dasb@math.iitb.ac.in}

\author[Debnath]{Ramlal Debnath}
\address{Department of Mathematics, Indian Institute of Technology Bombay, Powai, Mumbai, 400076, India}
\email{ramlal@math.iitb.ac.in, ramlaldebnath100@gmail.com}

\author[Panja]{Samir Panja}
\address{Department of Mathematics, Indian Institute of Technology Bombay, Powai, Mumbai, 400076, India}
\email{spanja@math.iitb.ac.in, panjasamir2020@gmail.com}



\begin{abstract}
We revisit the study of $\omega$-hypercontractions corresponding to a single weight sequence 
$\omega=\{\omega_k\}_{k\geq0}$ introduced by Olofsson in \cite{O} and find an analogue of Nagy-Foias characteristic function in this setting. Explicit construction of characteristic functions is obtained and it is shown to be a complete unitary invariant. By considering a multi-weight sequence $\clw$ and $\clw$-hypercontractions we extend Olofsson's work \cite{O} in the multi-variable setting. Model for $\clw$-hypercontractions is obtained by finding their dilations on certain weighted Bergman spaces over the polydisc corresponding to the multi-weight sequence $\clw$. This recovers and provides a different proof of the earlier work of Curto and Vasilescu \cite{CVPoly, CV} for $\gamma$-contractive multi-operators through a particular choice of multi-weight sequence.      

\end{abstract}

\maketitle

\section*{Notations}

\begin{list}{\quad}{}
\item $\mathbb N$\ \ \quad \quad Set of all natural numbers.
\item $\mathbb{Z}_+$ \quad \quad  Set of all positive integers.
\item $\mathbb{Z}^n_+$ \quad \quad $\{\alpha=(\alpha_1,\ldots, \alpha_n)\,|\, \alpha_i \in \mathbb{Z}_+, i=1,\ldots,n\}$.
\item 
$\mathbf e$\quad \quad \quad $(1,\dots,1)\in \Z_+^n$.
\item $\mathbb{R}_+$ \quad \quad  Set of all positive real numbers
including 0.
\item $\mathbb{R}_+^n$ \quad \quad $\{\gamma = (\gamma_1, \ldots, \gamma_n) \,|\, \gamma_i \in \mathbb{R}_+, i = 1,
\ldots, n\}$.
\item $\mathbb{C}^n$  \quad \quad Complex $n$-space.
\item $\bm{z}$ \quad \quad \; $(z_1, \ldots, z_n) \in \mathbb{C}^n$.
\item $\bm{z}^{\alpha}$ \quad \quad \,$z_1^{\alpha_1}\cdots
z_n^{\alpha_n}$ for all $\alpha\in\Z_+^n$.
\item $T$ \quad \quad \; $n$-tuple of commuting operators $(T_1, \ldots, T_n)$.
\item $T^{\alpha}$ \quad \quad $T_1^{\alpha_1} \cdots
T_n^{\alpha_n}$ for all $\alpha\in\Z_+^n$.
\item $\mathbb{D}^n$  \quad \quad Open unit polydisc $\{\z \,|\, |z_i|
<1, i=1, \ldots,n\}$.
\item $\clb(\clh)$ \quad Set of all bounded linear operators on a Hilbert space $\clh$.
\end{list}

\newsection{Introduction}

Dilations of operators on Hilbert spaces is a mathematical tool which is used to understand operators in terms of simple and well-understood operators. The basic idea of dilation of an operator $T$ on a Hilbert space $\clh$ is to find a well-understood operator $V$ on $\clk$ such that $T$ is a part of 
$V$, that is, $\clh \subseteq \clk$ is a $V^*$-invariant subspace and 
\[ T \cong P_{\clh}V|_{\clh} .\]
In a similar vain, by fixing a well-understood operator (more generally a class of well-understood operators), one can ask for a characterization of operators which are part of the fixed operator (or the class of operators). First significant result in this direction is due to Sz.-Nagy and Foias \cite{NF} which states that a contraction $T$ on a Hilbert space $\clh$ is pure (that is, $T^{*n} \to 0$ as $n \to \infty$ in the strong operator topology) if and only if $T$ is a part of the shift operator $M_z$ on a vector valued Hardy space $H^2_{\cle}(\D)$. Here for a Hilbert space $\cle$, 
\[ H^2_{\cle}(\D) = \{f: \D \to \cle \,|\, f(z)= \sum_{n \geq 0} a_n z^n, \sum_{n \geq 0} \|a_n\|^2 < \infty, z\in \D, a_n \in \cle \} \]
is the $\cle$-valued Hardy space over $\D$ and $M_z: H^2_{\cle}(\D) \to H^2_{\cle}(\D)$ is the shift operator, defined by, $(M_zf)(w) = wf(w)$ for all $w \in \D$. In other words, $T$ is a pure contraction if and only if there exist a Hilbert space $\cle$ and an $M_z^*$-invariant subspace $\clq \subseteq H^2_{\cle}(\D)$ such that 
\[ T \cong P_{\clq}M_z|_{\clq} .\]   
Another remarkable consequence of the Sz.-Nagy and Foias dilation result is that there exist a Hilbert space $\cle_*$ and a $\clb(\cle_*,\cle)$-valued inner multiplier 
$\theta_T$ (that is, $\theta_T:\mathbb D\to \clb(\cle_*,\cle)$ is a contractive analytic function such that $\theta_T$ is isometry-valued a.e. on the unit circle), known as the \textit{characteristic function} of $T$ (cf. \cite{NF}), such that 
\[ \clq = H^2_{\cle}(\D) \ominus \theta_T H^2_{\cle_*}(\D) .\]
The above dilation result is also extended for general contractions and is the stepping stone of Sz.-Nagy and Foias theory for contractions. 
Subsequently, by considering Bergman shift on vector-valued weighted Bergman space, Agler in his seminal paper \cite{AG} extended the above result. 
He showed that a contraction $T$ is a part of the Bergman shift $M_z$ acting on some  
$\cle$-valued weighted Bergman space $A^2_m(\cle)$ with kernel
\[ K_m(z,w) = (1 -z\bar{w})^{-m}I_{\cle}, \quad \quad (z,w \in \D) \]
 if and only if $T$ is a pure $m$-hypercontraction, 
that is, $T$ is a pure contraction and 
\[ K_m^{-1}(T,T^*):= \sum_{k=0}^m (-1)^k \binom{m}{k}T^kT^{*k} \geq 0. \]
 Recently, Olofsson in \cite{O} extended it further for shifts on weighted Bergman spaces corresponding to a certain class of weight sequences. He showed that if $\omega$ is a weight sequence then $T$ is part of the Bergman shift on some $\cle$-valued weighted Bergman space $A^2_{\omega}(\cle)$ if and only if $T$ is a pure $\omega$-hypercontraction. The reader is referred to Section 2 and Theorem 2.3 below for terminologies and a detailed description of the result. There are also several works in the multi-variable setting and an incomplete list of references is
 \cite{MV}, \cite{CV}, \cite{CVPoly}, \cite{AEM}, \cite{AE}, \cite{BDHS}, \cite{BDS1}, \cite{BD}, \cite{MR}, \cite{SHI} and \cite{Ball_Bolo}. 
 
 The purpose of the present article is twofold. 
 Firstly, we revisit the study of hypercontractions corresponding to a class of weight sequences as considered in ~\cite{O}; these hypercontractions are known as $\omega$-hypercontractions where $\omega$ is a weight sequence (see Section~\ref{Sec2} for the definition). Using dilations of such hypercontractions, we find an analogue of Sz.-Nagy and Foias characteristic functions in this setting.  Explicit construction of such characteristic functions is given in terms of triples which we call as characteristic triple. As expected, it is also shown that the characteristic function is a complete unitary invariant. This generalizes the work of ~\cite{BDS} and ~\cite{E} for the case $n=1$. The main ideas behind this consideration comes from a recent article ~\cite{BDS}. For recent developments on this topic in different context, the reader is referred to ~\cite{E}, ~\cite{BB}, \cite{MT},  \cite{JS1}, \cite{JS} and \cite{G}. Section~\ref{Sec2} is devoted to discuss these.  

 Secondly, we extend Olofsson's result \cite{O} in the polydisc setting. We say that $\clw=(\omega_1,\dots,\omega_n)$ is a multi-weight sequence if for all $i=1,\dots,n$, $\omega_i=\{\omega^{(i)}_m\}_{m\ge 0}$ is a weight sequence (see Definition 2.1 below for the definition of a weight sequence). We introduced the notion of $\clw$-hypercontractions corresponding to a multi-weight sequence $\clw$ and obtained their models by finding their dilations on some weighted Bergman space over the polydisc. Our method of multi-variable dilation is driven by the idea of using one variable dilation result, obtained by Olofsson, at a time and it is well supported by a commutant lifting type result obtained in this setting. 
 For a particular choice of multi-weight sequence $\clw$, we recover the dilation result of Curto and Vasilescu \cite{CVPoly} with a different proof.  
It is worth mentioning here that in the setting of unit ball in $\mathbb C^n$, Schilo (see Theorem 3.21 \cite{Sc}) extended Olofsson's results.  
 To describe our result succinctly we need to develop some notations and terminology. 
 
 Let $T=(T_1,\ldots,T_n) \in \clb(\clh)^n$ be an $n$-tuple of commuting contractions. For $\beta= (\beta_1,\ldots,\beta_n) \in \Z_+^n$ with $\beta\ge \mathbf e=(1,\dots,1)$ consider the multi-weight sequence  $\clw_{\beta}= (\omega_{\beta_1},\ldots, \omega_{\beta_n})$, where $\omega_{\beta_i}= \Big\{\frac{1}{{{\beta_i+l-1} \choose {l}}}\Big\}_{l\geq0}$ and $ {{\beta_i+l-1} \choose {l}}=\frac{(\beta_i+l-1)!}{(\beta_i-1)!l!}$, for all $i=1,\ldots,n$. Then for a Hilbert space $\cle$, $A^2_{\clw_{\beta}}(\cle)$ is the $\cle$-valued weighted Bergman space over $\D^n$ with kernel $K_{\clw_{\beta}}I_{\cle}$, where 
\[ K_{\clw_{\beta}}(\bm z, \bm w) = \prod_{i=1}^n \frac{1}{(1-z_i\bar{w}_i)^{\beta_i}}\quad (\bm z =(z_1,\ldots,z_n) \in \D^n, \bm w= (w_1,\ldots, w_n) \in \D^n). \] 
Set
\[ K_{\clw_{\beta}}^{-1}(T,T^*) = \sum_{0 \leq \alpha \leq \beta} (-1)^{|\alpha|} \frac{\beta!}{\alpha! (\beta - \alpha)!} T^{\alpha}T^{*\alpha} .\] 
Here for $\alpha = (\alpha_1,\ldots,\alpha_n), \beta= (\beta_1,\ldots,\beta_n) \in \Z_+^n, |\alpha| = \sum_{i=1}^n \alpha_i, \alpha! = \alpha_1!\cdots \alpha_n!$, and $\alpha \leq \beta$ if and only if $\alpha_i \leq \beta_i$ for all $i=1,\ldots,n$. We say that an $n$-tuple of commuting contraction is a part of the multi-shift 
$(M_{z_1},\ldots,M_{z_n})$ on $A^2_{\clw_{\beta}}(\cle)$ if there exists a joint $(M_{z_1}^*,\ldots,M_{z_n}^*)$-invariant subspace $\clq$ of $A^2_{\clw_{\beta}}(\cle)$ such that 
\[
T_i\cong P_{\clq}M_{z_i}|_{\clq}
\]
for all $i=1,\dots,n$. In this set up, Curto and Vasilescu \cite{CV} proved that for $\beta\in\mathbb Z_+^n$ with $\beta\ge \mathbf e$, a commuting tuple of contractions $T=(T_1,\ldots,T_n)$ is a part of the multi-shift $(M_{z_1},\ldots,M_{z_n})$ on $A^2_{\clw_{\beta}}(\cle)$ 
if and only if $T$ is pure and satisfies
\[ K_{\clw_{\beta}}^{-1}(T,T^*) \geq 0.\] 
In this article, 
we show that this result is true for a large class of kernels corresponding to multi-weight sequences. For a multi-weight sequence $\clw=(\omega_1,\dots,\omega_n)$, we set 
\[
S(\clw):=\{(\omega_{\lambda_1}',\dots,\omega_{\lambda_n}'): \omega_{\lambda_i}'\in\{\omega_{\lambda_i}, \mathds{1}\}, i=1,\dots,n\},
\]
where we denote by $\mathds 1$ the constant weight sequence 1. 
For a coefficient Hilbert space $\cle$, we denote by $A^2_{\clw}(\cle)$ the reproducing kernel Hilbert space corresponding to the kernel $K_{\clw}$ on the polydisc $\D^n$ defined by 
\[
K_{\clw}(\z,\w)=\sum_{\alpha\in\Z_+^n} \frac{1}{\omega^{(1)}_{\alpha_1}\cdots \omega^{(n)}_{\alpha_n}}(\z\bar{\w})^{\alpha}I_{\cle}.
\]
By one of the assumptions in the definition of multi-weight sequence, the analytic function 
\[
k_{\clw}(\z)=\sum_{\alpha\in\Z_+^n} \frac{1}{\omega^{(1)}_{\alpha_1}\cdots \omega^{(n)}_{\alpha_n}} \z^{\alpha}
\]
on $\D^n$ associated to $K_{\clw}$ does not vanish on $\D^n$. Suppose that 
\begin{equation}\label{inverse}
\frac{1}{k_{\clw}(\z)}= \sum_{\alpha\in\Z_+^n} c_{\alpha} \z^{\alpha} \quad (\z\in\D^n)
\end{equation}
is the Taylor expansion of $1/k_{\clw}$.
For an $n$-tuple of commuting contractions $T=(T_1,\ldots,T_n) \in \clb(\clh)^n$, and $ \br \in (0,1)^n$, using the hereditary functional calculus introduced by Agler in \cite{AG1}, we define  
\begin{equation}\label{D_{WT}}
D_{\clw,T}(\br):=\sum_{\alpha\in \mathbb{Z}^n_{+}}c_\alpha \br^{\alpha}T^{\alpha}T^{* \alpha}, 
\end{equation} where $c_{\alpha}$ as in (\ref{inverse})
and if $\br=(r_1,\dots,r_n)$, $\br^{\alpha}= r_1^{\alpha_1} \cdots r_n^{\alpha_n}$.

\begin{Definition}
An $n$-tuple of commuting contractions $T=(T_1,\ldots,T_n)\in \clb(\clh)^n$ is said to be an $\clw$-hypercontraction corresponding to a multi-weight sequence $\clw= (\omega_1, \ldots,\omega_n)$ if 
$ D_{\clw',T}(\br)\geq 0$ for all $\clw'\in S(\clw)$ and $\br \in (0,1)^n$. 
In addition, if each $T_i$ is a pure contraction, then we say that $T$ is a pure $\clw$-hypercontraction. 

\end{Definition}

With these terminologies, one of the main theorems of this article is the following.
\begin{Theorem}\label{main}
Let $\clw$ be a multi-weight sequence. An $n$-tuple of commuting contractions $T=(T_1,\dots,T_n)$ on $\clh$ is a part of the multi-shift $(M_{z_1},\dots,M_{z_n})$ on $A^2_{\clw}(\cle)$ for some Hilbert space $\cle$ if and only if $T$ is a pure $\clw$-hypercontraction. 
\end{Theorem}
This theorem is proved in Section~\ref{Sec4} as Theorem~\ref{model-gen-pure} by finding dilation of such an $n$-tuple of commuting contraction $T$. 
We also remove the pureness assumption and find dilations of $\clw$-hypercontractions (see Theorem~\ref{model-gen-tuple} for more details). 

In the case when $\clw=\clw_{\beta}$ for some $\beta\in \Z_+^n$ with $\beta\ge \mathbf e$, we show that $T$ is a pure $\clw_{\beta}$-hypercontractions if and only if $T$ is pure and $K_{\clw_{\beta}}^{-1}(T,T^*)\ge 0$. Therefore, for the particular choice of multi-weight sequence $\clw=\clw_{\beta}$, Theorem~\ref{main} recovers the classical result of Curto and Vasilescu \cite{CV, CVPoly} with a different proof. For the choice of multi-weight sequence $\clw_{\beta}$ when $\beta\in \mathbb R^n_+$ with $\beta\ge \mathbf e$, it also provides a natural generalization. Moreover, the class of multi-weight sequences is wide enough to include tensor product of reproducing kernel Hilbert spaces corresponding to certain Nevanlinna-Pick kernels over $\mathbb D$ (see Example ~\ref{Exam_1} below). 
Section~\ref{Sec3} is devoted to study multi-weight sequences and $\clw$-hypercontractions. In Section~\ref{Sec4}, we find dilations of pure $\clw$-hypercontractions and more generally for $\clw$-hypercontractions.

\newsection{Characteristic functions for $\omega$-Hypercontractions}\label{Sec2}
We construct characteristic functions of $\omega$-hypercontractions in this section. We recall the notion of  $\omega$-hypercontractions and their dilations first. Let $\omega=\{\omega_k\}_{k\ge 0}$ be a positive decreasing sequence such that $\omega_0=1$ and 
$\liminf_{k \to \infty} \omega_k^{\frac{1}{k}} = 1$.
Corresponding to the sequence $\omega$ and a Hilbert space $\cle$, we denote by $A^2_{\omega}(\cle)$ the $\cle$-valued weighted Bergman space; the Hilbert function space consists of $f\in\mathcal{O}(\D,\cle)$, the space of $\cle$-valued analytic functions $f$ on the open unit disc ($\D$), such that 
\[ f(z) = \sum_{k \geq 0} a_kz^k\ \text{ and } \|f\|^2_{\omega} := \sum_{k \geq 0} \|a_k\|^2 \omega_k < \infty \quad (a_k \in \cle,\,\, z\in \D).\]
It is also a reproducing kernel Hilbert space with the kernel $K_{\omega} : \D \times \D \to \clb(\cle)$ defined by  
\[ K_{\omega}(\zeta,\eta) = \sum_{k \geq 0} \frac{1}{\omega_k} 
(\zeta\bar{\eta})^k  I_{\cle} \quad \quad (\zeta,\eta \in \D).\]
If the co-efficient Hilbert space is $\mathbb C$, we simply write $A^2_{\omega}$ to denote $A^2_{\omega}(\mathbb C)$.
The multiplication operator $M_z$, known as shift operator, on $A^2_{\omega}(\cle)$ is defined by $(M_z f)(\eta)= \eta f(\eta)$ for all $\eta \in \D$. A straight forward computation shows that \[ \|M_z\|^2 = \text{sup}_k \frac{\omega_{k+1}}{\omega_k} .\] Thus an equivalent condition for the shift operator to be a contraction is that the sequence $\omega$ has to be a decreasing sequence. 
The kernel function $K_{\omega}$ has an associated analytic function $k_{\omega}$ on $\D$ defined by 
\[k_{\omega}(z)=\sum_{k \geq 0} \frac{1}{\omega_k} 
z^k.\] Properties of this associated analytic function has been very crucial in Olofsson's consideration. In fact, the class of weight sequences $\omega$ that he considered in \cite{O} are so that the associated analytic function $k_{\omega}$ on $\D$ possess some additional properties. The first natural property is that $k_{\omega}$ is non-vanishing. Then $\frac{1}{k_{\omega}}$ is also analytic on $\D$ and suppose that $\frac{1}{k_{\omega}}(z)= \sum_{k \geq 0} c_k z^k$. Other properties that $k_{\omega}$ needs to satisfy are as follows: 
\begin{enumerate} [(\textbf{P}1)]
\item The function $\frac{k_\omega}{k_{\omega,r}}$ has non negative Taylor coefficients for $0<r<1$, where $k_{\omega,r}(z)=k_\omega(rz)$.  
\item The quotients $\frac{k_{\omega,r}}{k_\omega}$ have uniformly bounded Taylor coefficients for $0<r<1$.
\item The Taylor coefficients of the reciprocal function $\frac{1}{k_{\omega}}$ is absolutely summable and the absolute sum of Taylor coefficients of $\frac{k_{\omega ,r}}{k_{\omega}}$
for $0 < r < 1$ form a uniformly bounded family.
\end{enumerate}
The first two properties are essential to obtain dilations of $\omega$-hypercontractions and we briefly indicate below the role played by these properties. The above discussion also prompt us to make the following definition. 
\begin{Definition}
A weight sequence is a positive decreasing sequence $\omega=\{\omega_n\}_{n\ge 0}$ such that $\omega_0=1$, $\liminf_{n \to \infty} \omega_n^{\frac{1}{n}} = 1$ and the corresponding analytic function $k_{\omega}$ is non-vanishing on $\D$ and satisfies (\textbf{P}1) and (\textbf{P}2) as above.
\end{Definition}
Natural examples of weight sequences are
the constant sequence $\omega_n = 1$ for all $n\ge 0$ and for a fixed $m\in \mathbb N$, $\omega_n = \frac{1}{\binom{n+m-1}{n}}$ for all $n \geq 0$. The constant sequence case corresponds to the Hardy space where as the later corresponds to the Bergman space defined above with kernel  
 $K_m(z,w) = \frac{1}{(1-z\bar{w})^m}$ ($z,w\in\D$).

{\em For the rest of this section we fix a weight sequence} $\omega$. We suppose that the reciprocal of the associated analytic function $k_{\omega}$ has the following power series expansion:
\begin{equation}\label{c_n}
\frac{1}{k_{\omega}}(z)= \sum_{n \geq 0} c_n z^n\quad (z\in\D).
\end{equation}

Now we recall the notion of $\omega$-hypercontraction introduced by Olofsson in \cite{O}.
\begin{Definition}\label{omega-hyp}
A bounded linear operator $T\in \clb(\clh)$ is said to be an $\omega$-hypercontraction if $T$ is a contraction and satisfies
\[ D_{\omega,T}(r):=\sum_{n\geq 0}r^n c_n T^n T^{*n} \geq 0\]
for all $r\in(0,1)$, where $c_n$'s are as in ~\eqref{c_n}. 
\end{Definition}
It can be shown that for the choice $\omega=\{\omega_n\}_{n\ge 0}$ where $\omega_n=\frac{1}{\binom{n+m-1}{n}}$ for all $n \geq 0$, $\omega$-hypercontractivity for a contraction $T$ is same as $m$-hypercontractivity in the sense of Agler~\cite{AG} (see \cite[Theorem 4.5]{O} for a proof). Thus the notion of $\omega$-hypercontractions is a natural generalization of $m$-hypercontractions. Moreover, it has been shown in ~\cite{O} that every $\omega$-hypercontraction is part of the shift operator $M_z$ on the weighted Bergman space $A^2_{\omega}(\cle)$ for some suitable Hilbert space $\cle$; we briefly recall this dilation result next. 

Let $T\in\mathcal{B}(\clh)$ be an $\omega$-hypercontraction. Then using the property (\textbf{P}1), it can be shown that  
the SOT limit of the operator $D_{\omega,T}(r)$ exists as $r\to 1$. We denote 
\[
D_{\omega,T}(1):=\textit{SOT}-\lim_{r\to1}D_{\omega,T}(r),
\]
and we define the defect operator and the defect space of $T$ as
\[
D_{\omega,T}:=\big (D_{\omega,T}(1) \big )^{1/2} \quad \text{and} \quad 
\cld_{\omega,T}:=\overline{\text{ran}}(D_{\omega,T}),
\] respectively. On the other hand, the property (\textbf{P}2) helps one to establish the identity 
\begin{equation}\label{omega-iso}
\|h\|^2 = \sum_{k \geq 0} \frac{1}{\omega_k}\|D_{\omega,T} T^{*k}h\|^2 + \lim_{k \to \infty} \|T^{*k}h\|^2 \quad \quad ( h\in \clh).
\end{equation}
Then it is evident from the above identity that the map $\pi_{\omega,T}: \clh \to A^2_{\omega}(\cld_{\omega,T})$ defined by 
\begin{equation}\label{pi_omega} \pi_{\omega,T}h(z) = 
\sum_{k \geq 0} \frac{1}{\omega_k}(D_{\omega,T} T^{*k}h) z^k \quad \quad (h\in \clh, z \in \D).
\end{equation} 
is a contraction and $\pi_{\omega,T}T^* = M_z^* \pi_{\omega,T}$. Moreover, setting  $Q_T^2:= \text{SOT}-\lim_{k \to \infty} T^k T^{*k}$ and $\clq_T:= \overline{\text{ran}}\, Q_T$, we have an isometry $\Pi_{\omega,T}: \clh \to A^2_{\omega}(\cld_{\omega,T}) \oplus \clq_T $ defined by 
\[(\Pi_{\omega,T}h)(z) = \Big((\pi_{\omega,T}h)(z), Q_Th \Big) \quad (h\in\clh, z\in\D) 
\]
and $\Pi_{\omega,T}T^* = (M_z \oplus U)^* \Pi_{\omega, T}$, where $U$ is a co-isometry on 
$\clq_T$ such that $U^*Q_Th=Q_T T^*h$ for all $h\in\clh$. 
Summarizing the above discussion we have the following dilation result.
\begin{Theorem}[c.f. ~\cite{O}]\label{olo}
Let $\omega=\{\omega_k\}_{k\geq0}$ be a weight sequence. 
If $T$ is an $\omega$-hypercontraction on $\clh$, then 
there exist an isometry $\Pi_{\omega,T}: \clh \to A^2_{\omega}(\cld_{\omega,T}) \oplus \clq_T$ and a co-isometry $U$ on $\clq_T$ such that 
\[
\Pi_{\omega,T}T^* = (M_z \oplus U)^* \Pi_{\omega, T}.
\]
In addition, if $T$ is pure 
then there exists an isometry $ \pi_{\omega, T}: \clh \to A^2_{\omega}(\cld_{\omega,T})$
such that
\[
\pi_{\omega,T}T^* = M_z^* \pi_{\omega, T}.
\]
\end{Theorem}
In the above theorem, the co-isometry $U$ can be made to a unitary by taking a co-extension. We do not include it in the statement as we shall use the present form of the theorem in later section. The observant reader might have noticed that we have not used (\textbf{P}3) to obtain the above dilation. But (\textbf{P}3) provides a necessary and sufficient condition for a contraction $T \in \clb(\clh)$ to be an $\omega$-hypercontraction. To be more precise, a contraction $T$ is $\omega$-hypercontraction if and only if $D_{\omega,T}(1) \geq 0$.    
For more details we refer the reader to ~\cite[Theorem 6.2]{O}. 
For Hilbert spaces $\cle_1$ and $\cle_2$, an
operator-valued analytic map $\theta : \D \to \clb(\cle_1, \cle_2)$ is a multiplier from $H^2_{\cle_1}(\D)$ to  $A^2_{\omega}(\cle_2)$ 
if $\theta f \in A^2_{\omega}(\cle_2)$ for all $f \in H^2_{\cle_1}(\D)$. We denote by $\clm(H^2_{\cle_1}(\D) , A^2_{\omega}(\cle_2))$, the space of all multipliers from $H^2_{\cle_1}(\D)$ to $A^2_{\omega}(\cle_2)$.
We also use $M_{\theta}$, for 
$\theta \in \clm(H^2_{\cle_1}(\D), A^2_{\omega}(\cle_2))$, to denote the associated multiplication operator by $\theta$, that is, 
\[ M_{\theta}f = \theta f \quad \quad (f \in H^2_{\cle_1}(\D)) .\]
A multiplier $\theta \in \clm(H^2_{\cle_1}(\D), A^2_{\omega}(\cle_2))$ is said 
to be partially isometric if $M_{\theta}$ is a partially isometric operator from $H^2_{\cle_1}(\D)$ to $A^2_{\omega}(\cle_2)$. Such a partially isometric multiplier naturally occur in the Beurling-Lax-Halmos type characterization of invariant subspaces of vector-valued weighted Bergman spaces (see \cite{JS}).
The characterization relevant for us is the following.




\begin{Theorem}[c.f. \cite{JS}]\label{BLH type}
If $\cls$ is an $M_z$-invariant subspace of $A^2_{\omega}(\cle_*)$, then there exist a Hilbert space $\cle$ and a partially isometric multiplier 
$\theta \in \clm(H^2_{\cle}(\D) ,  A^2_{\omega}(\cle_*))$ such that 
\[ \cls =\theta H^2_{\cle}(\D). \]  
\end{Theorem}
Now combining Theorem~\ref{olo} and Theorem~\ref{BLH type}, we can associate a partially isometric multiplier to a pure $\omega$-hypercontraction as follows.

\begin{Corollary}
Let $T\in \clb(\clh)$ be a pure $\omega$-hypercontraction. Then there exist a Hilbert space $\cle$ and a partially isometric multiplier 
$\theta \in \clm (H^2_{\cle}(\D),  A^2_{\omega}(\cld_{\omega, T}))$ such that 
\[ T\cong P_{\clq_{\theta}}M_z|_{\clq_{\theta}}, \]
where $\clq_{\theta}=\text{ran}\, \pi_{\omega, T}$, $\pi_{\omega, T}$ is the dilation map corresponding to $T$ as in Theorem~\ref{olo} and $\clq_{\theta}^\perp=\theta H^2_{\cle}(\D)$.
\end{Corollary}
In the rest of this section, we compute such a partially isometric multiplier explicitly corresponding to each pure $\omega$-hypercontraction.
The construction is based on a recently developed technique found in ~\cite{BDS} in the context of $m$-hypercontractions. 
Let $T\in \clb(\clh)$ be a pure $\omega$-hypercontraction. Since $T$ is a contraction, we define $D_T:=(I-TT^*)^{1/2}$. It follows from Lemma 3.6 in \cite{O} that for $h\in \clh$,
\begin{equation}\label{Ol1}
\|D_Th\|^2=\|D_{\omega,T}h\|^2+\sum_{k\geq 1}(\frac{1}{\omega_k}-\frac{1}{\omega_{k-1}})\|D_{\omega,T}T^{*k}h\|^2.
\end{equation}
Also, since $T$ is pure it follows from ~\eqref{omega-iso} that 
\begin{equation}\label{O2}
\|h\|^2=\sum_{n\geq 0}\frac{1}{\omega_n}\|D_{\omega,T}T^{*n}h\|^2.
\end{equation} 
Consider the map $C_{\omega,T}:\clh \to l^2(\mathbb{Z}_+,\cld_{\omega,T})$ defined by 
 \[
 C_{\omega,T}(h)=\{\sqrt{\rho_n}D_{\omega,T}T^{* n}h\}_{n \geq 0}, \quad  (h\in\clh)
 \]
 where $\rho_0=1$ and $\rho_n=\frac{1}{\omega_n}-\frac{1}{\omega_{n-1}} \geq 0$ for all $n\geq 1$.
Then by the identity (\ref{Ol1}), for $h\in \clh$, 
\begin{align*}
\|C_{\omega,T}h\|^2&=\sum_{n\geq 0}\rho_n\|D_{\omega,T}T^{*n}h\|^2\\
&=\|D_{\omega,T}h\|^2+\sum_{n\geq 1}(\frac{1}{\omega_n}-\frac{1}{\omega_{n-1}})\|D_{\omega,T}T^{*n}h\|^2\\
&=\|D_Th\|^2\leq \|h\|^2.
\end{align*}
Thus $C_{\omega,T}$ is a contraction. 
Now by identity (\ref{O2}),
\[I_\clh=\pi^*_{\omega,T} \pi_{\omega,T}=\sum_{n\geq 0}\frac{1}{\omega_n}T^nD^2_{\omega,T}T^{*n}, \]
and consequently, 
\begin{align*}
I_\clh -C^*_{\omega,T} C_{\omega,T}&=\sum_{n\geq 0}\frac{1}{\omega_n}T^nD^2_{\omega,T}T^{*n}-\sum_{n\geq 0}\rho_n T^nD^2_{\omega,T}T^{*n}\\
&=\sum_{n\geq 0}(\frac{1}{\omega_n}-\rho_n)T^nD^2_{\omega,T}T^{*n}\\
&=\sum_{n\geq 1}\frac{1}{\omega_{n-1}}T^nD^2_{\omega,T}T^{*n}\\
&=\sum_{n\geq 0}\frac{1}{\omega_n}T^{n+1}D^2_{\omega,T}T^{* n+1}\\
&=TT^*.
\end{align*}
This shows that the map 
$
X_T=\begin{bmatrix}
T^*\\ C_{\omega,T}\end{bmatrix}:\clh \to \clh \oplus l^2(\mathbb{Z}_+,\cld_{\omega,T})
$ defined by 
\[
X_T(h)=(T^*h, C_{\omega, T}h)\quad (h\in\clh)
\]
is isometry. By adding a Hilbert space $\cle$, if necessary, we get a unitary operator 
\[
U:=\begin{bmatrix} X_T & Y_T\end{bmatrix}:\clh \oplus \cle \to \clh \oplus l^2(\mathbb{Z}_+,\cld_{\omega,T}),
 \]
where $Y_T=U|_{\cle}:\cle \to \clh \oplus l^2(\mathbb{Z}_+,\cld_{\omega,T})$ is a contraction.
By setting $Y_T=\begin{bmatrix} B \\ D \end{bmatrix},$ where $B=P_\clh Y_T \in \clb(\cle, \clh)$ and $D=P_{l^2(\mathbb{Z}_+,\cld_{\omega,T})}Y_T \in \clb(\cle,l^2(\mathbb{Z}_+,\cld_{\omega,T}))$, we have the following result which will lead us to construct the characteristic function. 
\begin{Theorem} 
Let 
$T\in \clb(\clh)$ be a pure $\omega$-hypercontraction. Then the map $C_{\omega,T}:\clh \to l^2(\mathbb{Z}_+,\cld_{\omega,T})$ defined by 
 \[
 C_{\omega,T}(h)=\{\sqrt{\rho_n}D_{\omega,T}T^{* n}h\}_{n \geq 0}, \quad \text{where}\,\,\,\, \rho_0=1 \quad \text{and} \quad \rho_n=\frac{1}{\omega_n}-\frac{1}{\omega_{n-1}} \quad \text{for all}\,\, n\geq 1,
 \] is a contraction and there exist a Hilbert space $\cle$ and bounded operators $B\in \clb(\cle, \clh)$ and $D =\{D_n\}_{n=0}^{\infty} \in \clb(\cle,l^2(\mathbb{Z}_+,\cld_{\omega,T}))$ where each 
 $D_n \in \clb(\cle, \cld_{\omega,T}) $ such that  
 \[
\begin{bmatrix}T^* &B\\ C_{\omega,T}& D\end{bmatrix} : \clh \oplus \cle \raro
\clh \oplus l^2(\mathbb{Z}_+,\cld_{\omega,T}) 
\] 
 is unitary.
\end{Theorem}
The fact that each triple $(\cle, B, D)$ -- which appears in the above theorem -- gives rise to a characteristic function of $T$, motivates us to make the following definition. 
\begin{Definition}
A triple $(\cle, B, D)$ consisting of a Hilbert space $\cle$ and bounded linear operators $B\in \clb(\cle, \clh)$ and $D\in \clb(\cle,l^2(\mathbb{Z}_+,\cld_{\omega,T}))$ is said to be a characteristic triple of a pure $\omega$-hypercontraction $T$ on $\clb(\clh)$ if
\[
\begin{bmatrix}T^* &B\\ C_{\omega,T}& D\end{bmatrix} : \clh \oplus \cle \raro
\clh \oplus l^2(\mathbb{Z}_+,\cld_{\omega,T}) 
\] 
 is unitary.
\end{Definition}
It turns out that characteristic triple is unique in the following sense.
\begin{Theorem} \label{ch-triple-unique}
If $(\cle_1,B_1,D_1)$ and $(\cle_2,B_2,D_2)$ are two characteristic triple of a pure $\omega$-hypercontraction $T \in \clb(\clh)$, 
then there exists a unitary $U: \cle_2 \to \cle_1$ such that 
\[ (\cle_2,B_2,D_2) = (U^* \cle_1, B_1U, D_1U). \]
\end{Theorem}
\textit{Proof.} The proof follows from the observation that 
$ \begin{bmatrix}
B_1 \cr D_1
\end{bmatrix}$ and $ \begin{bmatrix}
B_2 \cr D_2
\end{bmatrix}$ are isometries and their range is same. \qed

We are now in a position to state the main result of this section which provides an explicit method to construct characteristic functions. The proof of the theorem is similar to Theorem 3.1 in \cite{BDS} and we only include a sketch of the proof here. 

\begin{Theorem}\label{chfw}
Let $T$ be a pure $\omega$-hypercontraction on $\clh$, and let $(\cle, B, D)$ be a characteristic triple of $T$. Then
\[
\theta_T(z)=\sum_{n\geq 0}\sqrt{\rho_n} D_n z^n + z D_{\omega,T} \sum_{n\geq 0}\frac{1}{\omega_n}z^n T^{* n }B \]
is a partially isometric multiplier in 
$ \clm(H^2_{\cle}(\D),A^2_\omega(\cld_{\omega,T}))$
such that 
\[
\clq_T^\perp=\theta_T H^2_{\cle}(\D)\ \text{ and }
T\cong P_{\clq_T}M_z|_{\clq_T},
\]
where 
 $\rho_0=1, \,\,  
\rho_n=\frac{1}{\omega_n}-\frac{1}{\omega_{n-1}}$ for all $n\geq 1$.
\end{Theorem}
{\em Sketch of the proof.}
For a contraction $A$ and $z\in\D$, we set 
$K_{\omega}(z,A):= \sum_{n\geq 0}\frac{1}{\omega_n}z^n A^{ n }$, where the series converges as $zA$ is a strict contraction. 
Now, note that 
\[ (1 - zT^*) K_{\omega}(z,T^*) = \frac{1}{\omega_0} + \sum_{n \geq 1} (\frac{1}{\omega_n} - \frac{1}{\omega_{n-1}}) z^n T^{*n} = \sum_{n\geq 0} {\rho_n}z^n T^{* n }.\]
A direct calculation using the above identity and the unitary property of  $\begin{bmatrix}
T^* & B \cr C_{\omega,T}  &  D 
\end{bmatrix}$, 
it can be shown that 
\begin{equation}\label{key identity}
K_{\omega}(\eta,\zeta) I_{\cld_{\omega,T}} - \frac{\theta_T(\eta)\theta_T(\zeta)^*}{1 - \eta \bar{\zeta}} = D_{\omega,T}K_{\omega}(\eta,T^*)K_{\omega}(\bar{\zeta},T)D_{\omega,T}\quad (\zeta,\eta\in\D).
\end{equation}
Then using some standard arguments in the theory of reproducing kernel Hilbert spaces, we conclude that \[ \theta_T \in \clm(H^2_{\cle}(\D), A^2_{\omega}(\cld_{\omega,T})) \] and
\[ M_{\theta_T}^*K_{\omega}(.,\zeta)h = K_1(.,\zeta)\theta_T(\zeta)^*h \quad  (\zeta \in \D, h \in \cld_{\omega,T}), \] where $K_1(\eta,\zeta) = (1 - \eta \bar{\zeta})^{-1} $. 
Consequently, 
\[ (I - M_{\theta_T}M_{\theta_T}^*) K_{\omega}(.,\zeta)h = (K_{\omega}(.,\zeta)I_{\cld_{\omega,T}} - K_1(.,\zeta)\theta_T(.)\theta_T(\zeta)^*)h, \] 
and therefore by ~\eqref{key identity},
\[ (I - M_{\theta_T}M_{\theta_T}^*) K_{\omega}(.,\zeta)h = D_{\omega,T}K_{\omega}(.,T^*)K_{\omega}(\bar{\zeta}, T)D_{\omega,T}h \]
for all $\zeta \in \D$ and $h \in \cld_{\omega,T}$. On the other hand, the adjoint of the dilation map for the pure $\omega$-hyeprercontraction $\pi_{\omega,T}^*: A^2_{\omega}(\cld_{\omega,T}) \to \clh$ is given by 
\[ \pi_{\omega,T}^* K_{\omega}(.,\zeta)h = K_{\omega}(\bar{\zeta}, T) D_{\omega,T}h \quad \quad (\zeta \in \D, h \in \cld_{\omega,T}). \] 
Then it is easy to see that
\[ \pi_{\omega,T}\pi_{\omega,T}^* K_{\omega}(.,\zeta)h = D_{\omega,T}K_{\omega}(.,T^*)K_{\omega}(\bar{\zeta}, T)D_{\omega,T}h, \] for all $\zeta \in \D$ and $h \in \cld_{\omega,T}$. Combining all these we have the required identity \[ \pi_{\omega,T}\pi_{\omega,T}^* + M_{\theta_T} M_{\theta_T}^* = I_{A^2_{\omega}(\cld_{\omega,T})}. \] That is, $\theta_T$ is a partially isometric multiplier such that $\clq_T^{\perp} = \theta_T H^2_{\cle}(\D)$. \quad \qed \\

We call the partially isometric multiplier $\theta_T$ obtained in the above theorem corresponding to a   characteristic triple $(\cle, B,D)$ as \textit{characteristic function} of $T$. 
Let $T_1$ and $T_2$ be two pure $\omega$-hypercontractions 
on $\clh_1$ and $\clh_2$, respectively. Also, let $\theta_{T_1}$ and $\theta_{T_2}$ be characteristic functions corresponding to the characteristic triples $(\cle_1,B_1,D_1)$ and $(\cle_2,B_2,D_2)$ of $T_1$ and $T_2$, respectively. Then the characteristic functions $\theta_{T_1}$ and $\theta_{T_2}$ is said to be \textit{coincide} 
if there exists two unitaries $\tau: \cle_2 \to \cle_1$ and $\tau_*: \cld_{\omega,T_1^*} \to \cld_{\omega,T_2^*}$ such that 
\[ \theta_{T_2}(z) = \tau_* \theta_{T_1}(z) \tau \quad \quad (z \in \D). \]
We end the section with the observation that characteristic function is completely unitary invariant for pure $\omega$-hypercontractions. We omit the proof as it is exactly same as Theorem 3.2 in \cite{BDS}.
\begin{Theorem}
Let $T_1$ and $T_2$ be two pure $\omega$-hypercontractions
on $\clh_1$ and $\clh_2$, respectively. Then $T_1$ and $T_2$ are unitary equivalent if and only if $\theta_{T_1}$ and $\theta_{T_2}$ coincide.
\end{Theorem}

\newsection{ $\clw$-Hypercontractions }\label{Sec3}

Let, for each $i=1,\ldots ,n$, $\omega_i=\{\omega^{(i)}_{m}\}_{m \geq0}$ be a weight sequence. Then the $n$-tuple of weight sequences $\clw = (\omega_1,\ldots, \omega_n)$ is called a multi-weight sequence. Corresponding to such a multi-weight sequence $\clw$ and a Hilbert space $\cle$, consider the $\cle$-valued weighted Bergman space $A^2_{\clw}(\cle)$ over $\D^n$ consists of $f\in \mathcal{O}(\D^n,\cle)$ such that 
\[
f(z)=\sum_{\alpha\in \mathbb{Z}^n_{+}}a_{\alpha}\z^{\alpha}\ \text{ and }
\|f\|^2:=\sum_{\alpha\in\mathbb{Z}^n_{+}}\|a_{\alpha}\|^2\omega_{\alpha_1}^{(1)}\cdots\omega_{\alpha_n}^{(n)}<\infty\quad (a_{\alpha}\in\cle, \z\in\D^n).
\]
The Hilbert space $A^2_{\clw}(\cle)$ is unitarily equivalent to $A^2_{\omega_1}\otimes\cdots\otimes  A^2_{\omega_n}\otimes \cle$ via the natural unitary which sends the monomials $\z^{\alpha}e$ to $z_1^{\alpha_1}\otimes \cdots \otimes z_n^{\alpha_n}\otimes e$ for all $\alpha\in\mathbb Z_{+}^n$ and $e\in\cle$; it is also a reproducing kernel Hilbert space with kernel 
\[
K_{\clw}(\z,\w)=K_{\omega_1}(z_1,w_1)\cdots K_{\omega_n}(z_n,w_n) I_{\cle}= \sum_{\alpha\in\mathbb{Z}^n_{+}} \frac{1}{\omega_{\alpha_1}^{(1)}\cdots \omega_{\alpha_n}^{(n)}}(\z\bar{\w})^{\alpha}
I_{\cle}\quad (\z,\w\in\D^n).
\]
The analytic function associated to $K_{\clw}$ is crucial in what follows and has the form 
\[
k_{\clw}(\z)=k_{\omega_1}(z_1)\cdots k_{\omega_n}(z_n)= \sum_{\alpha\in \mathbb{Z}^n_{+}}\frac{1}{\omega_{\alpha_1}^{(1)}\cdots \omega_{\alpha_n}^{(n)}} \z^{\alpha} \quad (\z\in\D^n),
\]
where $k_{\omega_i}(z) = \sum_{m \geq0}\frac{1}{\omega^{(i)}_m}z^{m}$ is the associated analytic function corresponding to the kernel function $K_{\omega_i}$ and the weight sequence $\omega_i$.
Since each $\omega_i$ is a weight sequence, $k_{\omega_i}$ does not vanish on $\D$ and consequently $k_{\clw}$ does not vanish on $\D^n$ and 
\begin{align}\label{inverse_kernel}
\frac{1}{k_{\clw}(\z)}=\sum_{\alpha\in \mathbb{Z}^n_{+}}c_{\alpha_1}^{(1)}\cdots c_{\alpha_n}^{(n)} \z^\alpha , 
\end{align}
where the coefficients satisfy 
\[
\frac{1}{k_{\omega_i}(z)}=\sum_{m\geq0}c^{(i)}_{m}z^{m} \quad (z\in \mathbb{D}, 1\le i\le n).
\]
The several variable analogue of properties (\textbf{P}1), (\textbf{P}2) and (\textbf{P}3) that the analytic function $k_{\clw}$ should satisfy are as follows:
\begin{enumerate} [($\textbf{P}1'$)]
\item The function $\frac{k_\clw}{k_{\clw,\br}}$ has non negative Taylor coefficients for all $\br\in(0,1)^n$, where for $\br=(r_1,\dots,r_n)$ and $\z\in\D^n$, \[k_{\clw,\br}(\z)=k_{\omega_1}(r_1z_1)\cdots k_{\omega_n}(r_nz_n).\]  
\item The quotients $\frac{k_{\clw,\br}}{k_\clw}$ have uniformly bounded Taylor coefficients for all $\br\in(0,1)^n$.
\item The Taylor coefficients of the reciprocal function $\frac{1}{k_{\clw}}$ is absolutely summable and the absolute sum of Taylor coefficients of $\frac{k_{\clw ,\br}}{k_{\clw}}$
for all $\br\in(0,1)^n$ form a uniformly bounded family.
\end{enumerate}

We observe that the analytic function $k_{\clw}$ automatically inherits properties ($\textbf{P}1'$) and ($\textbf{P}2'$). 
\begin{Proposition}\label{Prop_Pro}
Let $\clw=(\omega_1,\dots,\omega_n)$ be a multi-weight sequence and $k_{\clw}$ be the associated analytic function on $\D^n$ as above. Then $k_{\clw}$ satisfies ($\textbf{P}1'$) and ($\textbf{P}2'$). 

In addition if, for all $i=1,\dots,n$, $k_{\omega_i}$ satisfies (\textbf{P}3) then $k_{\clw}$ also satisfies ($\textbf{P}3'$).
\end{Proposition}
\begin{proof}
Let $\br, \mathbf s\in (0,1]^n$. Then for each $i$ and $|z_i|< \text{min} (1/r_i, 1/s_i)$, we have
\[
\frac{k_{\omega_i}(r_i z_i)}{k_{\omega_i}(s_i z_i)}=(\sum_{n\geq0}\frac{1}{\omega^{(i)}_{n}}r_i^{n} z_i^{n})(\sum_{m\geq0}c^{(i)}_{m}s_i^{m} z_i^{m})=\sum_{m\geq0}a^{(i)}_{m}(r_i,s_i) z_i^{m},
\]
where 
\begin{equation}\label{asr}
a^{(i)}_{m}(r_i,s_i)=\sum_{0 \leq k \leq m}\frac{r_i^{k}}{\omega^{(i)}_{k}}c^{(i)}_{{m-k}} s_i^{m-k}.
\end{equation}
Then 
\begin{align*}
\frac{k_{\clw,\br}(\z)}{k_{\clw,\mathbf s}( \z)} 
&=\frac{k_{\omega_1}(r_1 z_1)}{k_{\omega_1}(s_1 z_1)}\cdots \frac{k_{\omega_n}(r_n z_n)}{k_{\omega_n}(s_n z_n)}
=\sum_{\alpha\in \mathbb{Z}^n_+}a_\alpha (\br,\mathbf s)\z^\alpha,
\end{align*}
where $a_\alpha (\br,\mathbf s)=a^{(1)}_{\alpha_1}(r_1,s_1)\cdots a^{(n)}_{\alpha_n}(r_n,s_n) $ for all $\alpha \in\mathbb{Z}^n_+$. This in particular shows that $k_{\clw}$ inherits ($\textbf{P}1'$) and ($\textbf{P}2'$) from that of $k_{\omega_i}$.

Now if each $k_{\omega_i}$ satisfies (\textbf{P}3), then it follows from ~\eqref{inverse_kernel} and the identity above that $k_{\clw}$ satisfies ($\textbf{P}3'$). This completes the proof. 
\end{proof}
Although there are abundant examples of analytic functions on $\D^n$ satisfying ($\textbf{P}1'$), ($\textbf{P}2'$) and ($\textbf{P}3'$). We provide a few examples corresponding to a certain type of natural multi-weight sequence.
\begin{Example}\label{Exam_1}
\textup{$(i)$} For any $\beta=(\beta_1, \ldots, \beta_n)\in \mathbb{R}^n_+ $ with $\beta\ge \mathbf e$,
we consider the multi-weight sequence $\clw_{\beta}= (\omega_{\beta_1},\ldots, \omega_{\beta_n})$, where $\omega_{\beta_i}= \Big\{\frac{1}{{{\beta_i+l-1} \choose {l}}}\Big\}_{l\geq0}$ and $ {{\beta_i+l-1} \choose {l}}=\frac{\Gamma(\beta_i+l)}{\Gamma{(\beta_i){l!}}}$, for $i=1,\ldots,n$. Then the analytic function on $\D^n$ corresponding to $\clw_{\beta}$ is 
\begin{align*}
k_{\clw_\beta}(\z) 
&= k_{\omega_{\beta_1}}(z_1) \cdots k_{\omega_{\beta_n}}(z_n) 
= \frac{1}{{(1-z_1)}^{\beta_1}} \cdots \frac{1}{{(1-z_n)}^{\beta_n}}\\
&= \Big(\sum_{\alpha_1\geq0}{{\beta_1+\alpha_1-1} \choose {\alpha_1}}z_1^{\alpha_1} \Big)  \cdots \Big(\sum_{\alpha_n\geq0}{{\beta_n+\alpha_n-1} \choose {\alpha_n}} z_n^{\alpha_n} \Big).
\end{align*}
By Corollary 5.5 in \cite{O}, the analytic function $k_{\omega_{\beta_i}}(z_i)=\frac{1}{(1-z_i)^{\beta_i}}$, for each $i=1,\ldots,n$, has the properties (\textbf{P}1), (\textbf{P}2), and (\textbf{P}3). Therefore, using Proposition \ref{Prop_Pro}, the analytic function $k_{\clw_\beta}(\z)$ on $\D^n$ corresponding to the multi-weight sequence $\clw_{\beta}$ also satisfies the properties ($\textbf{P}1'$), ($\textbf{P}2'$) and ($\textbf{P}3'$). 

Moreover, the reproducing kernel Hilbert space corresponding to the multi-weight sequence $\clw_{\beta}$ is the weighted Bergman space over $\D^n$ with kernel 
\[K_{\clw_{\beta}}(\z,\w)=\prod_{i=1}^n\frac{1}{(1-z_i\bar{w_i})^{\beta_i}}\quad (\z,\w\in\D^n).
\]
In particular, it is easy to see that for $\beta=(1,\dots,1)$, each of the weight sequence becomes the constant sequence $1$, which we denote by $\mathds{1}$, and the corresponding reproducing kernel Hilbert space is the Hardy space over $\D^n$. 

\textup{$(ii)$} For each $i=1,\ldots,n$, we consider the sequence $\omega_i=\{\omega^{(i)}_{k}\}_{k \geq0}$ such that it is a positive decreasing sequence satisfying $\omega^{(i)}_0=1$, $\liminf_{n \to \infty} \omega_k^{{(i)}^{\frac{1}{k}}} = 1$, and the associated analytic function $k_{\omega_i}(z)=\sum_{l \geq 0} \frac{1}{\omega^{(i)}_{l}} 
z^{l}$ ($z\in \D$) is a finite product of analytic functions corresponding to some Nevanlinna-Pick kernels \cite{AgMc}. Then by Proposition 5.4 in \cite{O}, each $\omega_i=\{\omega^{(i)}_{k}\}_{k \geq0}$ is a weight sequence and the associated analytic function $k_{\omega_i}$ satisfies (\textbf{P}3). Therefore, using Proposition \ref{Prop_Pro}, $\clw=(\omega_1,\ldots,\omega_n)$ is a multi-weight sequence and the associated analytic function $k_{\clw}$ satisfies ($\textbf{P}3'$). 

\end{Example}
 Recall that an $n$-tuple of commuting contractions $T=(T_1,\ldots,T_n)\in \clb(\clh)^n$ is an $\clw$-hypercontraction corresponding to a multi-weight sequence $\clw= (\omega_1, \ldots,\omega_n)$ if 
$ D_{\clw',T}(\br)\geq 0$ for all $\clw'\in S(\clw)$ and $\br \in (0,1)^n$,
where 
\begin{equation}\label{D_{WT}}
D_{\clw,T}(\br):=\sum_{\alpha\in \mathbb{Z}^n_{+}}c_\alpha \br^{\alpha}T^{\alpha}T^{* \alpha}, 
\end{equation}  
$c_{\alpha}$ as in (\ref{inverse_kernel})
and
\[
S(\clw):=\{(\omega_{\lambda_1}',\dots,\omega_{\lambda_n}'): \omega_{\lambda_i}'\in\{\omega_{\lambda_i}, \mathds{1}\}, i=1,\dots,n\}.
\]
For a non-empty subset $\Lambda=\{\lambda_1,\dots, \lambda_m\}$ of $I=\{1,\dots,n\}$, we set 
\[
T_{\Lambda}:=(T_{\lambda_1},\dots,T_{\lambda_m}),\ \text{ and } \clw_{\Lambda}:=(\omega_{\lambda_1},\dots,\omega_{\lambda_m}).
\]




One of the important properties is that $\clw$-hypercontractivity is preserve for subtuples of an $\clw$-hypercontraction. 

\begin{Proposition}\label{subtuple}
Let $T$ be an $\clw$-hypercontraction for some multi-weight sequence $\clw$. Then for any non-empty subset $\Lambda\subseteq I$, $T_{\Lambda}$ is a $\clw_{\Lambda}$-hypercontraction.
\end{Proposition}
\begin{proof}
We only consider $\Lambda= \{1,\ldots,n-1\}$ and show that $T_{\Lambda}$ is an $\clw_{\Lambda}$-hypercontraction, as the argument for general $\Lambda\subseteq I$ is similar. Let 
$\clw'=(\omega_1',\ldots,\omega_{n-1}')\in S(\clw_{\Lambda})$ and set $\clw'': = ( \clw',\mathds{1}) = (\omega'_1, \ldots, \omega'_{n-1},\mathds{1})\in S(\clw)$.
Since $T=(T_1,\ldots,T_n)$ is an $\clw$-hypercontraction and $\clw''\in S(\clw)$,  $D_{\clw'', T}(\br) \geq 0$ for all $\br \in (0,1)^n$. \\
Then using Agler's hereditary calculus to the identity 
\[ k_{\clw''}(r_1z_1,\ldots,r_nz_n) = (1-r_nz_n)^{-1} k_{\clw'}(r_1z_1,\ldots,r_{n-1}z_{n-1}), \] 
we get 
\[ D_{\clw'',T}(\br) = D_{\clw',T_{\Lambda}}(\br') - r_nT_n D_{\clw',T_{\Lambda}}(\br')T_n^* ,\]
where $\br'=(r_1,\dots,r_{n-1})$.
Now using telescoping sum we have
\[ D_{\clw',T_{\Lambda}}(\br^\prime) = \sum_{k=0}^l r_n^kT_n^k D_{\clw'',T}(\br) T_n^{*k} + r_n^{l+1}T_n^{l+1}D_{\clw',T_{\Lambda}}(\br^\prime) T_n^{*(l+1)}.\]
Since $r_n\in(0,1)$, taking limit as $l\to\infty$
we conclude that 
\[ D_{\clw',T_{\Lambda}}(\br^\prime) = \sum_{k=0}^{\infty}r^k_n T_n^k D_{\clw'',T}(\br) T_n^{*k}\ge 0.
\]
This completes the proof.
\end{proof}

The following proposition is the key to define defect operator and defect space corresponding to an $\clw$-hypercontraction. 

\begin{Proposition}\label{Pro_Decr}
Let $\clw$ be a multi-weight sequence. Let $T$ be an $n$-tuple of commuting contractions such that $D_{\clw, T}(\br)\ge 0$ for all $\br \in (0,1)^n$. Then 
\[\text{SOT}-\lim_{\br \to \mathbf e} D_{\clw, T}(\br)
\]
is a positive operator, where $\mathbf e=(1,\dots,1)$.

Moreover, if the associated analytic function $k_\clw$ satisfies ($\textbf{P}3'$) then 
\[
\lim_{\br \to \mathbf e} D_{\clw, T}(\br)=\sum_{\alpha\in \mathbb{Z}^n_+}c_\alpha T^\alpha T^{* \alpha},
\]
where the sum converges in the operator norm in $\clb(\clh)$.
\end{Proposition}
\begin{proof}
For the existence of the strong operator limit, it is enough to show that if $\br,\mathbf s\in (0,1)^n$ with $\br \leq \mathbf s$ then $D_{\clw,T}(\mathbf s)\le D_{\clw,T}(\br)$. Let $ \br \leq \mathbf s $. Without any loss of generality, we assume that $r_i=s_i$ for all $i=1,\dots,m$ and $r_i< s_i$ for all $i=m+1,\dots,n$. Then using the identity 
\begin{align*}
\frac{k_{\clw}(s_1 z_1)\cdots k_{\clw}(s_nz_n)}{k_{\clw}(r_1 z_1)\cdots k_{\clw}(r_nz_n)} 
&
= \frac{k_{\clw}(s_{m+1} z_{m+1})\cdots k_{\clw}(s_nz_n)}{k_{\clw}(r_{m+1} z_{m+1})\cdots k_{\clw}(r_nz_n)}
=\sum_{\alpha\in \mathbb{Z}^{n-m}_+}a_\alpha ( s_{m+1},\dots,s_n,r_{m+1},\dots r_n)\z^\alpha,
\end{align*}
where $a_\alpha ( s_{m+1},\dots,s_n,r_{m+1},\dots r_n)=a^{(m+1)}_{\alpha_1}(s_{m+1},r_{m+1})\cdots a^{(n)}_{\alpha_{n-m}}(s_n,r_n) $ for all $\alpha \in\mathbb{Z}^{n-m}_+$, 
and Agler's hereditary functional calculus we have that 
\begin{equation}\label{D_iden}
D_{\clw,T}(\br)=\sum_{\alpha\in \mathbb{Z}^{n-m}_{+}}a_\alpha ( s_{m+1},\dots s_n, r_{m+1},\cdots, r_n) T^{\alpha} D_{\clw,T}(\mathbf s) T^{* \alpha}. 
\end{equation} 
Moreover, since $(\omega_{m+1},\dots,\omega_n)$ is a multi-weight sequence, by ($\textbf{P}1'$), \[a_\alpha ( s_{m+1},\dots s_n, r_{m+1},\cdots, r_n)\ge 0\] for all $(\alpha\in\mathbb Z^{n-m}_{+})$ and $a_{(0,\dots,0)} ( s_{m+1},\dots s_n, r_{m+1},\cdots, r_n)=1$. This shows that $D_{\clw,T}(\mathbf s)\le D_{\clw, T}(\br)$.

Finally, if $k_\clw$ satisfies ($\textbf{P}3'$), by a simple use of Lebesgue's dominated convergence theorem we also have 
$ D_{\clw, T}(\br)\to \sum_{\alpha\in \mathbb{Z}^n_+}c_\alpha T^\alpha T^{* \alpha}$ as $\br\to\mathbf{e}$ in $\clb(\clh)$.
This completes the proof.
\end{proof}

\begin{Remark}\label{split limit}
The proof of the above proposition also suggests that if $\br \in (0,1)^m$ and $\mathbf s\in (0,1)^{n-m}$ then 
\[
\text{SOT}-\lim_{\br \to \mathbf (1,\dots,1)} D_{\clw, T}(\br,\mathbf s)
\geq 0.
\]

\end{Remark}
 We denote the positive operator whose existence is shown in the above proposition by 
 \[
 D_{\clw,T}(\mathbf e):=\text{SOT}-\lim_{\br \to \mathbf e} D_{\clw, T}(\br),
 \]
 and the defect operator and the defect space of $T$ by
\begin{equation}\label{defect}
 D_{\clw,T}:=D_{\clw,T}(\mathbf e)^{1/2} \text{ and }
 \cld_{\clw,T}:=\overline{\text{ran}} D_{\clw,T},
\end{equation}
respectively.

It could be difficult in general to determine when an $n$-tuple of commuting contractions $T=(T_1,\ldots,T_n)$ is an $\clw$-hypercontraction as it asks to verify infinitely many inequalities. 
However, for certain multi-weight sequences it becomes easier to verify.
For instance, consider the multi-weight sequence $\clw_{\gamma}$ for $\gamma=(\gamma_1,\ldots, \gamma_n)\in \Z^n_+$ with $\gamma\ge \mathbf e$ as in ~\ref{Exam_1}. In this case, we show that the notion of $\gamma$-contractive multi-operator in the sense of ~\cite{CV} is same as $\clw_{\gamma}$-hypercontraction. Recall that an $n$-tuple of commuting contractions $T=(T_1,\dots,T_n)$
is a $\gamma$-contractive multi-operator if for all $0\le \beta\le \gamma$,
\[
K_{\clw_{\beta}}^{-1}(T,T^*)=\sum_{0\le \alpha\le \beta}\frac{\beta!}{\alpha!(\beta-\alpha)!}T^{\alpha}T^{*\alpha}\ge 0.
\]
Observe that the operator $K_{\clw_\beta}^{-1}(T,T^*)$ can also be represented as 
\[
K_{\clw_{\beta}}^{-1}(T,T^*)=(I-C_{T_1})^{\beta_1}\cdots (I-C_{T_n})^{\beta_n}(I),
\]
where for an operator $A\in\clb(\clh)$, the completely positive map $C_A:\clb(\clh) \to \clb(\clh)$ is defined by $C_A(X)=A X A^*$ ($X\in \clb(\clh)$).
We begin with the following lemma.
\begin{Lemma}\label{intermediate}
Let $\gamma=(\gamma_1,\dots,\gamma_n)\in \mathbb R_{+}^n$ and $\br\in (0,1)^n$. If $ T=(T_1,\ldots ,T_n)\in \clb(\clh)^n$ satisfies $D_{\clw_{\gamma}, T}(\br)\ge 0$  then $D_{\clw_{\beta},T}(\br)\ge 0$ for all $\beta\in\mathbb R_+^n$ such that 
 $0 \leq \beta \leq \gamma$.
\end{Lemma}

\begin{proof}
Let $0\le \beta\le \gamma$.
For $(z_1,\dots,z_n)\in\mathbb D^n$ and $\br\in(0,1)^n$, applying Agler's hereditary functional calculus to the identity
 \[
(1-r z_1)^{\beta_1} \dots (1-r z_n)^{\beta_n} =\frac{1}{(1-r z_1)^{\gamma_1 - \beta_1} \ldots (1-r z_n)^{\gamma_n - \beta_n}} (1- r z_1)^{\gamma_1} \dots (1- r z_n)^{\gamma_n},\]
we have
\begin{equation*}
D_{\clw_{\beta},T}(\br)=\sum_{\delta \in \Z^n_+} \br^{\delta} {{\gamma - \beta+ \delta -\mathbf e} \choose \delta} T^\delta D_{\clw_{\gamma},T}(\br) 
T^{* \delta}\ge 0,
\end{equation*}
where for $\beta=(\beta_1,\dots,\beta_n)$ and $\alpha=(\alpha_1,\dots,\alpha_n)$, ${\beta \choose \alpha}={\beta_1 \choose \alpha_1}\cdots {\beta_n \choose \alpha_n}$.
This completes the proof.
\end{proof}

To simplify computations we borrow the following notation from ~\cite{CV}. For $\beta\in\mathbb{Z}_+^n$, we denote by $\Delta^{\beta}_T$ the map on $\clb(\clh)$ defined by 
\[\Delta^\beta_T:=(I-C_{T_1})^{\beta_1}\cdots (I-C_{T_n})^{\beta_n}\]
and for $\br=(r_1, \ldots, r_n) \in (0,1)^n$,
\[\Delta^\beta_{\br T}:=(I-r_1C_{T_1})^{\beta_1}\cdots (I-r_n C_{T_n})^{\beta_n}.\]

\begin{Theorem}\label{equivalence 1}
Let $T=(T_1,\ldots ,T_n)\in \clb(\clh)^n$ be an $n$-tuple of commuting contractions and  $\gamma \in \Z^n_+$ such that  $\gamma \geq \mathbf e$. Then the following are equivalent.
\begin{enumerate}[(a)]
    \item  
$T$ is an $\gamma$-contractive multi-operator.
    \item 
$T$ is an $\clw_{\gamma}$-hypercontraction.
\end{enumerate}
\end{Theorem}
\begin{proof}
If $T$ is an $\clw_{\gamma}$-hypercontraction then by 
Lemma~\ref{intermediate}, it follows that 
$D_{\clw_{\beta},T}(\br)\ge 0$ for all $0\le \beta\le \gamma$ and $\br\in(0,1)^n$. Then by taking limit as $\br\to \mathbf e$ and using  Proposition~\ref{Pro_Decr}, we conclude that $D_{\clw_{\beta},T}(\mathbf e)=K_{\clw_{\beta}}^{-1}(T,T^*)\ge 0$ for all $0\le \beta\le \gamma$. This proves $(b)\Rightarrow (a)$.

To prove $(a)\Rightarrow (b)$, let $0\le \beta\le \gamma$. We work on one component at a time and only show that $\Delta_{(T',rT_n)}^{\beta}(I)\ge 0$ for all $r\in(0,1)$, where $T'=(T_1,\dots,T_{n-1})$, as repetition of the same argument will then establish $\Delta_{\br T}^{\beta}(I)\ge 0$ for all $\br\in(0,1)^n$. To this end, we show that $\Delta^{\beta}_{(T',rT_n)}(I)\ge \Delta^{\beta}_{T}(I)\ge 0$. If $\beta=(\beta_1,\dots,\beta_n)$ and $\beta_n=0$, then there is nothing to prove. Assume that $\beta_n>0$ and set $\beta'=(\beta_1,\dots,\beta_{n-1})$. First note that
\begin{align*}
\Delta^{\beta}_{(T^\prime, r T_n)}(I)-\Delta^{\beta}_{T}(I)& =\Delta^{\beta'}_{T'}[ (I-rC_{T_n})^{\beta_n}(I)-(I-C_{T_n})^{\beta_n}(I)]\\
&=(1-r)\sum_{l=0}^{\beta_n-1}\Delta^{\beta'}_{T'}C_{T_n}(I-rC_{T_n})^{(\beta_n-1-l)}(I-C_{T_n})^l(I).
\end{align*}
It is now enough to show that  $\Delta^{\beta'}_{T'}(I-rC_{T_n})^k(I- C_{T_n})^{(\beta_n-1-k)}(I)\geq 0$ for all $k=0,\dots, \beta_n-1$. For $k=0$, $\Delta^{\beta'}_{T'}(I- C_{T_n})^{(\beta_n-1)}(I)= \Delta_{T}^{(\beta',\beta_n-1)}(I)\ge 0$ by the hypothesis. For $k=1$,     
\begin{align*}
    &\Delta^{\beta'}_{T'}(I-r C_{T_n})(I- C_{T_n})^{(\beta_n-2)}(I)\\
    &=\Delta^{\beta'}_{T'}[(I- C_{T_n})^{(\beta_n-2)}(I)-r C_{T_n}(I- C_{T_n})^{(\beta_n-2)}(I)]\\
    &=\Delta^{(\beta^\prime, \beta_n-2)}_{T}(I)-r C_{T_n} \Delta^{(\beta^\prime, \beta_{n}-2)}_{T}(I)\\
     &\ge  \Delta^{(\beta^\prime, \beta_n-2)}_{T}(I)- C_{T_n} \Delta^{(\beta^\prime, \beta_{n}-2)}_{T}(I)\quad (\text{as }\, \Delta^{(\beta^\prime, \beta_n-2)}_{T}(I)\ge 0 \,\text{and}\, 0<r<1)\\ 
    &=\Delta^{(\beta',\beta_n-1)}_{T}(I).
\end{align*}
This shows that $\Delta^{\beta'}_{T'}(I-r C_{T_n})(I- C_{T_n})^{(\beta_n-2)}(I)\ge \Delta^{(\beta',\beta_n-1)}_{T}(I)\ge 0 $. By a similar calculation as above one can show that 
\[
\Delta^{\beta'}_{T'}(I-r C_{T_n})^2(I- C_{T_n})^{(\beta_n-3)}(I)\ge \Delta^{\beta'}_{T'}(I-r C_{T_n})(I- C_{T_n})^{(\beta_n-2)}(I)\ge  \Delta^{(\beta',\beta_n-1)}_{T}(I)\ge 0. 
\]
Repeating this $k$-times we have the following chain of inequalities:
\begin{align*}
&\Delta^{\beta'}_{T'}(I-r C_{T_n})^k(I- C_{T_n})^{(\beta_n-1-k)}(I)\ge \Delta^{\beta'}_{T'}(I-r C_{T_n})^{k-1}(I- C_{T_n})^{(\beta_n-k)}(I)\\
&\ge \cdots \ge \Delta^{\beta'}_{T'}(I-r C_{T_n})(I- C_{T_n})^{(\beta_n-2)}(I)\ge  \Delta^{(\beta',\beta_n-1)}_{T}(I)\ge 0. 
\end{align*}
This completes the proof.
\end{proof}

Even for $\gamma\in \mathbb{R}^n_+$, we have a similar result to the above Theorem.
\begin{Theorem}
Let $T=(T_1,\ldots ,T_n)\in \clb(\clh)^n$ be an $n$-tuple of commuting contractions and  $\gamma \in \mathbb{R}^n_+$ such that  $\gamma \geq \mathbf e$. Then the following are equivalent.
\begin{enumerate}[(a)]
    \item  
$K_{\clw_{\beta}}^{-1}(T,T^*):=(I-C_{T_1})^{\beta_1}\cdots (I-C_{T_n})^{\beta_n}(I)\geq 0$ for all $\beta\in \mathbb{R}^n_+$ such that $0\leq \beta \leq  \gamma$.
    \item 
$T$ is an $\clw_{\gamma}$-hypercontraction.
\end{enumerate}
\end{Theorem}

\begin{proof}
If $T$ is an $\clw_{\gamma}$-hypercontraction then by 
Lemma~\ref{intermediate}, it follows that 
$D_{\clw_{\beta},T}(\br)\ge 0$ for all $\beta\in \mathbb{R}^n_+$ such that $0\leq \beta \leq  \gamma$. Then by taking limit as $\br\to \mathbf e$ and using  Proposition~\ref{Pro_Decr}, we conclude that $D_{\clw_{\beta},T}(\mathbf e)=K_{\clw_{\beta}}^{-1}(T,T^*)\ge 0$ for all $0\le \beta\le \gamma$. This proves $(b)\Rightarrow (a)$.

To prove $(a)\Rightarrow (b)$, it is enough to prove that for $0\le \beta\le\gamma$ and $r\in (0,1)$, $\Delta_{(T',rT_n)}^{(\beta^\prime, \beta_n)}(I)\ge 0$, where $\beta=(\beta_1,\dots,\beta_{n})$ and $\beta'=(\beta_1,\dots,\beta_{n-1})$.
Let $r\in (0,1)$ be fixed. If $\beta_n=0$, then $\Delta_{(T',rT_n)}^{(\beta^\prime, \beta_n)}(I)=\Delta_{T'}^{\beta^\prime}(I)\ge 0$. Assume that $\beta_n>0$. First we consider the case when $\beta_n\in \Z_+$. If $1\le \beta_n$, then using the fact $\Delta^{\beta'}_{T'}(I)\ge 0$, we have 
\[
\Delta_{(T',rT_n)}^{(\beta^\prime, 1)}(I)=\Delta_{T'}^{\beta^\prime}(I-rC_{T_n})(I)\ge \Delta_{T'}^{\beta^\prime}(I-C_{T_n})(I)=\Delta_{T}^{(\beta',1)}(I)\ge 0.
\]
Similarly if $2\le \beta_n$, then using $\Delta_{(T',rT_n)}^{(\beta^\prime, 1)}(I)\ge 0$, 
\[
\Delta_{(T',rT_n)}^{(\beta^\prime, 2)}(I)=\Delta_{(T',rT_n)}^{(\beta^\prime,1)}(I-rC_{T_n})(I)\ge \Delta_{(T',rT_n)}^{(\beta^\prime,1)}(I-C_{T_n})(I)=\Delta_{T}^{(\beta',1)}(1-rC_{T_n})(I)\ge \Delta^{(\beta',2)}_T(I)\ge 0.
\]
By repeating this method sufficient number of times we conclude that $\Delta_{(T',rT_n)}^{(\beta^\prime, \beta_n)}(I)\ge 0$ if $\beta_n$ is an integer. Next we suppose that $\beta_n\notin \Z_+$.  
Let $[\beta_n]$ be the largest integer which is less than or equal to $\beta_n$. We set $\delta=\beta_n-[\beta_n]$. Since $0< \delta <1$, observe that $(1-x)^{\delta}=1- \sum_{k=1}^{\infty} b_k x^k$, where $b_k\ge 0$ for all $k\ge 1$.  
Now using $\Delta_{(T',rT_n)}^{(\beta^\prime, [\beta_n])}(I)\ge 0$,
\begin{align*}
  &\Delta_{(T',rT_n)}^{(\beta^\prime, \beta_n)}(I)\\
  &=\Delta^{\beta^\prime}_{T^\prime}(I-r C_{T_n})^{\beta_n}(I)\\
  &= \Delta^{\beta^\prime}_{T^\prime}(I-r C_{T_n})^{[\beta_n]}(I-r C_{T_n})^{\delta}(I)\\
  &=\Delta_{(T',rT_n)}^{(\beta^\prime, [\beta_n])}(I)- \sum_{k=1}^{\infty}b_k r^k C_{T_n}^k\Delta_{(T',rT_n)}^{(\beta^\prime, [\beta_n])}(I)\\
  &\ge  \Delta_{(T',rT_n)}^{(\beta^\prime, [\beta_n])}(I)- \sum_{k=1}^{\infty}b_k C_{T_n}^k\Delta_{(T',rT_n)}^{(\beta^\prime, [\beta_n])}(I)\\
  & = \Delta_{(T',rT_n)}^{(\beta^\prime, [\beta_n])}(I- C_{T_n})^{\delta}(I).
\end{align*}
To complete the proof of the theorem, we apply the similar strategy as in the proof of the above theorem to get a chain of inequalities. That is,
\begin{align*}
 &\Delta^{\beta^\prime}_{T^\prime}(I- C_{T_n})^{\delta}(I-r C_{T_n})^{[\beta_n]}(I)\geq \Delta^{\beta^\prime}_{T^\prime}(I- C_{T_n})^{\delta+1}(I-r C_{T_n})^{[\beta_n]-1}(I)\\
 &\geq \cdots \geq \Delta^{\beta^\prime}_{T^\prime}(I- C_{T_n})^{\delta+[\beta_n]-1}(I-r C_{T_n})(I) \geq \Delta^{\beta^\prime}_{T^\prime}(I- C_{T_n})^{\delta+[\beta_n]}(I)= \Delta^{\beta}_{T}(I)\geq0.   
\end{align*}
This completes the proof.
\end{proof}

We end this section with a few examples of $\clw$-hypercontractions.
\begin{Examples} \label{examples}
\textup{$(i)$} We say $T=(T_1,\dots,T_n)$ is a Szeg\"{o} tuple on $\clh$ if $T$ is a commuting tuple of contractions on $\clh$ such that $\mathbb{S}_n^{-1}(T,T^*)\ge 0$ where $\mathbb{S}_n$ is the Szeg\"{o} kernel of the Hardy space over $\D^n$. 
Moreover, we say $T$ is a Brehmer tuple (\cite{Breh}) if $T_{\Lambda}$ is a Szeg\"{o} tuple for any non-empty subset $\Lambda$ of $I$. If $T$ is a Brehmer tuple then, by Theorem~\ref{equivalence 1}, $T$ is an $\clw$-hypercontraction for the multi-weight sequence $\clw=(\mathds{1},\dots,\mathds{1})$.

\textup{(ii)} Let $T=(T_1,\dots,T_n)$ be an $n$-tuple of commuting co-isometries on $\clh$ and $\clw=(\omega_1,\ldots,\omega_n)$ be a multi-weight sequence. For any  $\clw'\in S(\clw)$, we have 
 \begin{align*}
    D_{\clw',T}(\br)= \sum_{\alpha\in \mathbb{Z}^{n}_{+}}c^\prime_\alpha \br^{\alpha}T^{\alpha}T^{*\alpha} = \frac{1}{k_{\clw^\prime}(\br)}I_{\clh}  \quad \quad ( \br \in (0,1)^n ).
\end{align*}
 Since, $k_{\clw^\prime}(\br) > 0 $ for any $\clw'\in S(\clw)$, so $T$ is an $\clw$-hypercontraction for any multi-weight sequence $\clw$.
 

\textup{(iii)} Let $R= (R_1,\ldots,R_n)$ be an $n$-tuple of commuting contractions on $\clh$ such that $R_1$ is an co-isometry and $R^{'}= (R_2,\ldots,R_n)$ is an $\clw'$-hypercontraction. Then for any weight sequence $\omega_1$, consider the multi-weight sequence $\clw=(\omega_1,\clw')$. Let $\clw''=(\omega_1,\tilde{\clw})\in S(\clw) $. Then clearly $\tilde{\clw}\in S(\clw')$ and observe that for $\br\in(0,1)^n$,
\begin{align*}
    D_{\clw'',R}(\br)
    &=\sum_{\alpha \in \Z_+^{n}}  
    c^{(1)}_{\alpha_1}\cdots c^{(n)}_{\alpha_n} \br^{\alpha} R_1^{\alpha_1} 
    R_{2}^{\alpha_2} \cdots R_{n}^{\alpha_n} R_1^{* \alpha_1}R_{2}^{* \alpha_2} \cdots R_{n}^{* \alpha_n} \\
    &=\sum_{\alpha \in \Z_+^{n}}  
    c^{(1)}_{\alpha_1}\cdots c^{(n)}_{\alpha_n} \br^{\alpha}  
    R_{2}^{\alpha_2} \cdots R_{n}^{\alpha_n}
    R_1^{\alpha_1}R_1^{* \alpha_1}R_{2}^{* \alpha_2} \cdots R_{n}^{* n} \\
    &= \big( \sum_{\alpha_1\geq0} c^{(1)}_{\alpha_1} r_{1}^{\alpha_1} \big) 
    D_{\tilde{\clw},R^{'}}(\br^{'})\\
    &= \frac{1}{k_{\omega_1}(r_1)}D_{\tilde{\clw},R^{'}}(\br^{'})\ge 0,
\end{align*}
where  $\br^{'}= (r_2,\ldots,r_{n})$. Thus  $R=(R_1,\ldots,R_n)$ is an $\clw$-hypercontraction on $\clh$.  

\textup{(iv)} 
Let $M_{\z}=(M_{z_1}, \ldots, M_{z_n})$ be the $n$-tuple of multi-shifts on $A^2_{\clw}(\cle)$ for some multi-weight sequence $\clw$, that is
\[M_{z_i}f(\w)=w_i f(\w) \quad (i=1,\ldots,n, \quad \text{and}\quad \w \in \D^n).\]
Note that for $f(z)=\sum_{\alpha\in \mathbb{Z}^n_{+}}a_{\alpha}\z^{\alpha}$ and $\beta=(\beta_1, \ldots, \beta_n)\in \mathbb{Z}_+^n$,  
\[
M^{* \beta}_{\z}f(\w)=\sum_{\alpha \in\Z_+^n} \frac{\omega_{\alpha_1+\beta_1}^{(1)}\cdots\omega_{\alpha_n+\beta_n}^{(n)}}{\omega_{\alpha_1}^{(1)}\cdots\omega_{\alpha_n}^{(n)}} a_{\alpha +\beta} \w^\alpha.\]
Then
\[
 \|M^{* \beta}_{\z}f\|^2=\sum_{\alpha \in\Z_+^n} \frac{\omega_{\alpha_1+\beta_1}^{{(1)}^2}\cdots\omega_{\alpha_n+\beta_n}^{{(n)}^2}}{\omega_{\alpha_1}^{(1)}\cdots\omega_{\alpha_n}^{(n)}} \|a_{\alpha+\beta}\|^2.\]
Now, for $\bm{r}=(r_1,\ldots,r_n)\in (0,1)^n$, 
\begin{align*}
    \langle  D_{\clw,M_{\z}}(\bm{r}) f, f \rangle&=\sum_{\beta\in \mathbb{Z}^n_{+}}c_\beta \bm{r}^{\beta} \|M_{\z}^{* \beta}f\|^2\\
    &=\sum_{\beta \in \mathbb{Z}^n_{+}}c_\beta \bm{r}^{\beta}( \sum_{\alpha \in\Z_+^n} \frac{\omega_{\alpha_1+\beta_1}^{{(1)}^2}\cdots\omega_{\alpha_n+\beta_n}^{{(n)}^2}}{\omega_{\alpha_1}^{(1)}\cdots\omega_{\alpha_n}^{(n)}} \|a_{\alpha+\beta}\|^2)\\
    &=\sum_{\alpha \in \mathbb{Z}^n_{+}} \omega_{\alpha_1}^{{(1)}^2}\cdots\omega_{\alpha_n}^{{(n)}^2}( \sum^\alpha_{\beta \geq 0} \frac{c_\beta \bm{r}^{\beta}}{\omega_{\alpha_1-\beta_1}^{(1)}\cdots\omega_{\alpha_n-\beta_n}^{(n)}}) \|a_{\alpha}\|^2\\
    &=\sum_{\alpha \in \mathbb{Z}^n_{+}} \omega_{\alpha_1}^{{(1)}^2}\cdots\omega_{\alpha_n}^{{(n)}^2}a_\alpha(\mathbf e, \bm{r})\|a_{\alpha}\|^2,
\end{align*}
where for the last equality we use identities as in ~\ref{asr}. By the property ($\textbf{P}1'$), $a_\alpha(\mathbf e, \bm{r})\geq 0$ for all $\bm{r}\in (0,1)^n$ and therefore,  $D_{\clw,M_{\z}}(\bm{r})\geq 0$  for all $\bm{r}\in (0,1)^n$. 
Thus the $n$-tuple of multi-shifts $M_{\z}$ on $A^2_{\clw}(\cle)$ is an $\clw$-hypercontraction. In fact, $M_z$ is a pure $\clw$-hypercontraction. To see this first observe that the space $\text{span}\{\z^\alpha h: \alpha\in \Z_+^n, h\in \cle\}$ is dense in $A^2_{\clw}(\cle)$. For fixed $\alpha\in\Z_+^n$ and $ h\in \cle$, if we consider $k\in\Z_+$ such that $k>|\alpha|$, then $M_{z_i}^{* k}(\z^\alpha h)=0$. Finally, since $\{M_{z_i}^{*k}\}_{k\ge 1}$ is uniformly bounded we get that $M_{z_i}^{* k}\to 0$ in the strong operator topology. This shows that $M_{z_i}$ is a pure contraction for all $i=1,\dots,n$. 
\end{Examples}

\newsection{Model For $\clw$-hypercontractions}\label{Sec4}

The main purpose of this section is to find dilations of $\clw$-hypercontractions. This multi-variate dilation is obtained using one-variable dilation at a time. The key to use one variable dilation theory is a commutant lifting result which we describe first. Let $T=(T_1,\dots,T_n)$ be an $\clw$-hypercontraction on $\clh$ corresponding to a multi-weight sequence $\clw=(\omega_1,\dots,\omega_n)$. Then considering the subset $\Lambda=\{i\}$ of $I$ and using Proposition~\ref{subtuple} we have that each $T_i$ is an $\omega_i$-hypercontraction for all $i=1,\dots,n$. Thus $D_{\omega_i, T_i}(1):=\text{SOT}-\lim_{r\to 1}D_{\omega_i,T_i}(r)\ge 0$ for all $i=1,\dots,n$, and recall that the corresponding defect operator and defect space are 
\[
D_{\omega_i, T_i}= (D_{\omega_i, T_i}(1))^{1/2}\ \text{ and } \cld_{\omega_i,T_i} = \overline{ran}\, D_{\omega_i,T_i},
\]
respectively. In what follows, we denote by $\hat{\clw_i}$ $(1\le i\le n)$ the multi-weight sequence obtained from $\clw=(\omega_1,\dots,\omega_n)$ by deleting $\omega_i$, that is, 
\[
\hat{\clw_i}:=(\omega_1,\dots,\omega_{i-1},\omega_{i+1},\dots,\omega_n).
\]


\begin{Proposition}\label{n-1 tuple lift}
Let $\clw$ be a multi-weight sequence and let $ T=(T_1,\ldots ,T_n)$ be an $\clw$-hypercontraction on $\clh$. 
Suppose $(M_z \oplus U_{T_1})$ on $A^2_{\omega_1}(\cld_{\omega_1,T_1}) \oplus \clq_{T_1}$ is the dilation of $T_1$ with the dilation map 
$\Pi_{\omega_1,T_1}: \clh \to A^2_{\omega_1}(\cld_{\omega_1,T_1}) \oplus \clq_{T_1}$ as obtained in Theorem~\ref{olo}. Then there exists an $\hat{\clw_1}$-hypercontraction 
$V=(V_2,\ldots,V_n)$ on $ A^2_{\omega_1}(\cld_{\omega_1,T_1}) \oplus \clq_{T_1}$ such that  
\[ \Pi_{\omega_1,T_1} T_i^* = V_i^* \Pi_{\omega_1,T_1}\ 
\text{ and }\
V_i(M_z \oplus U_{T_1})=(M_z \oplus U_{T_1})V_i,
\quad \quad (i=2,\ldots,n)\]  
where
$V_i= (I_{A^2_{\omega_1}} \otimes A_i) \oplus X_i$ ($i=2,\ldots,n$) for some commuting operator tuples $(A_2,\ldots,A_n)$ on $\cld_{\omega_1,T_1}$ and 
$(X_2,\ldots,X_n)$ on $\clq_{T_1}$.
\end{Proposition}

\textit{Proof:} 
Let $2\le i\le n$ and consider the subset $\Lambda=\{1,i\}$ of $I$. Since $T_{\Lambda}$ is a $\clw_{\Lambda}$-hypercontraction, then for the multi-weight sequence $\clw'=(\omega_1, \mathds{1})\in S(\clw_{\Lambda})$, we have $D_{\clw',T_{\Lambda}}(1,1)\ge 0$, that is, $D_{\omega_1,T_1}(1)-T_iD_{\omega_1,T_1}(1)T^*_i \geq 0$. 
Applying Douglas factorization lemma (\cite{D}) to the above inequality we have a contraction $A_i$ on $\cld_{\omega_1,T_1}$ such that 
\begin{equation}\label{intertwining A}
D_{\omega_1,T_1}T^*_i=A^*_i D_{\omega_1,T_1}.
\end{equation}
Thus, we get an $(n-1)$-tuple of commuting contractions $A=(A_2, \ldots , A_n)$ on $\cld_{\omega_1,T_1}$. We now show that $A$ is an $\hat{\clw_1}$-hypercontraction. To this end, we only consider the multi-weight sequence $\hat{\clw_1}$ and show that $D_{\hat{\clw_1},A}(\br)\ge 0$ for all $\br\in (0,1)^{n-1}$. As the required positivity corresponding to other multi-weight sequences in $S(\hat{\clw_1})$ can be shown similarly.  
For any $\br\in (0,1)^{n-1}$ and $h\in \clh$, we have 
\begin{align*}
\langle D_{\hat{\clw_1},A}(\br) D_{\omega_1,T_1}h,
D_{\omega_1,T_1}h\rangle
=&\sum_{\alpha=(\alpha_2,\dots,\alpha_n)\in\Z_{+}^{(n-1)}}
\br^{\alpha}c^{(2)}_{\alpha_2}\cdots c^{(n)}_{\alpha_n} 
\langle D_{\omega_1,T_1} A^{\alpha}
A^{* \alpha}D_{\omega_1,T_1}h,h \rangle\\
=& \sum_{\alpha=(\alpha_2,\dots,\alpha_n)\in\Z_{+}^{(n-1)}}
\br^{\alpha}c^{(2)}_{\alpha_2}\cdots c^{(n)}_{\alpha_n} 
\langle T^{\alpha}D_{\omega_1,T_1}(1) T^{* \alpha}h,h \rangle\\ 
=& \sum_{\alpha=(\alpha_2,\dots,\alpha_n)\in\Z_{+}^{(n-1)}}
\br^{\alpha}c^{(2)}_{\alpha_2}\cdots c^{(n)}_{\alpha_n} 
\lim_{s\to 1}\langle T^{\alpha}D_{\omega_1,T_1}(s) T^{* \alpha}h,h \rangle\\
=& \lim_{s\to 1}\sum_{\alpha\in\Z_{+}^{n}} (\br,s)^{\alpha}c_{\alpha} \langle T^{\alpha} T^{* \alpha}h,h\rangle\\
=&\lim_{s\to 1}\langle D_{\clw,T}(\br,s) h,h\rangle\ge 0.
\end{align*}
Here the positivity in the last equality follows from Remark~\ref{split limit}. 
This proves that $A$ is an $\hat{\clw_1}$-hypercontraction.

On the other hand, recall from the construction of dilation in Theorem~\ref{olo} that $Q^2_{T_1}= \text{SOT}- \lim_{n\to \infty} T_1^nT_1^{*n}$, $\clq_{T_1}=\overline{ran} Q_{T_1}$ and the co-isometry $U_{T_1}$ on $\clq_{T_1}$ is defined by the identity 
$U_{T_1}^*Q_{T_1}h=Q_{T_1}T_1^*h$ for all $h\in\clh$. Now for any $2\le i\le n$, since $T_i $ is a contraction we have 
 $T_iQ^2_{T_1}T^*_i\leq Q^2_{T_1}$. Again applying Douglas factorization lemma to the inequality we get a contraction $X_i$ on $\clq_{T_1}$ such that 
\begin{equation}\label{intertwining X}
X^*_iQ_{T_1}=Q_{T_1}T^*_i \quad (i=2,\ldots,n).
\end{equation}
It is now easy to see that $(U_{T_1}, X_2, \ldots , X_n)$ is an $n$-tuple of commuting contractions on $\clq_{T_1}$. As before, we show that $X=(X_2,\dots,X_n)$ is an $\hat{\clw_1}$-hypercontraction and for this we only show that $D_{\hat{\clw_1},X}(\br)\ge 0$ for all $\br\in (0,1)^{n-1}$ as the argument for other multi-weight sequences is similar.  
Now for $h\in \clh$ and $\br\in(0,1)^{n-1}$,  
\begin{align*}
 \langle D_{\hat{\clw_1},X}(\br)Q_{T_1}h,Q_{T_1}h \rangle 
 &=\sum_{\alpha=(\alpha_2,\dots,\alpha_n)\in\Z_{+}^{(n-1)}} \br^{\alpha} c^{(2)}_{\alpha_2} \cdots c^{(n)}_{\alpha_n}\langle T^{\alpha}(\lim_{k \to \infty}T^k_1T^{* k}_1)T^{*\alpha}  h,h \rangle\\
&=\lim_{k\to \infty}\sum_{\alpha=(\alpha_2,\dots,\alpha_n)\in\Z_{+}^{(n-1)}} \br^{\alpha} c^{(2)}_{\alpha_2} \cdots c^{(n)}_{\alpha_n}\langle T^{\alpha} T^{*\alpha}  T^{*k}_1 h,T^{* k}_1h \rangle\\
&= \lim_{k\to\infty}\langle D_{\hat{\clw_1},T}(\br)T_1^{* k}h,T_{1}^{* k}h \rangle\ge 0.
\end{align*}
Thus $X $ is an
$\hat{\clw_1}$-hypercontraction. 
We set  
\[ V_i:=(I_{A^2_{\omega_1}}\otimes A_i)\oplus X_i \in  \clb \big(A^2_{\omega_1}(\cld_{\omega_1,T_1})\oplus \mathcal{Q}_{T_1} \big)\quad (i=2,\dots,n).\] 
Then it is evident that $V=(V_2, \ldots , V_n)$ is an $\hat{\clw_1}$-hypercontraction. 
It remains to verify that $V$ satisfies the required commuting and intertwining relations. For any $h \in \clh$ and $i = 2, \ldots , n$,
\begin{align*}
\Pi_{\omega_1,T_1} T_i^* h  &= \big( \sum_{k \geq 0} z^k \frac{1}{\omega_k} D_{\omega_1,T_1} T_1^{*k} T_i^*h , Q_{T_1}T_i^*h \big)  \\
&= \big( \sum_{k \geq 0} z^k \frac{1}{\omega_k} A_i^* D_{\omega_1,T_1} T_1^{*k} h , X_i^*Q_{T_1}h \big)  \\ 
&= \big( (I_{A^2_{\omega_1}} \otimes A_i^*) \sum_{k \geq 0} z^k \frac{1}{\omega_k}D_{\omega_1,T_1} T_1^{*k} h , X_i^*Q_{T_1}h \big)  \\
&= V_i^* \Pi_{\omega_1, T_1} h.
\end{align*}
Thus $\Pi_{\omega_1,T_1}T_i^* = V_i^* \Pi_{\omega_1,T_1}$ for all $i = 2, \ldots , n$. Finally, as each $X_i$ commutes with $U_{T_1}$ it follows that 
$M_z \oplus U_{T_1}$ commutes with each $V_i$. This completes the proof.  \qed

\begin{Remark}\label{pure}
If $T_1$ is a pure contraction in the above proposition, then $Q_{T_1}=0$ and therefore $X_i=0$ for all $i=2,\dots,n$. Thus, in this case, the $\hat{\clw_1}$-hypercontraction $V=(V_2, \ldots , V_n)$ will be of the form 
$V_i = I_{A^2_{\omega_1}} \otimes A_i$ for all $i = 2, \ldots , n$.
\end{Remark}
The following lemma is needed to prove the general dilation result below. 
\begin{Lemma}\label{useful lemma}
Let $T, X,A$ be as in Proposition~\ref{n-1 tuple lift} and let $\Lambda\subseteq \{2,\dots,n\}$. Then \begin{enumerate}
 \item[\textup{(i)}] $D_{\omega_1,T_1}D_{\clw_{\Lambda},A_{\Lambda}}(1,\dots,1)D_{\omega_1,T_1}=D_{\clw_{\Lambda\cup\{1\}}, T_{\Lambda\cup\{1\}}}(1,\dots,1)$.
 \item[\textup{(ii)}]
 $Q_{T_1}D_{\clw_{\Lambda},X_{\Lambda}}(1,\dots,1)Q_{T_1}=
 \text{SOT}-\lim_{k \to \infty}T_1^k D_{\clw_{\Lambda}, T_{\Lambda}}(1,\dots,1)T_1^{* k}$.
\end{enumerate}
\end{Lemma}
\textit{Proof:} 
Let $\Lambda=(\lambda_1,\dots,\lambda_m)\subseteq \{2,\dots,n\}$. Then 
\begin{align*}
D_{\omega_1,T_1}D_{\clw_{\Lambda},A_{\Lambda}}(1,\dots,1)D_{\omega_1,T_1} & =D_{\omega_1,T_1}\big(\text{SOT}-\lim_{\br\to (1,\dots,1)}\sum_{\alpha\in\Z_+^{m}}\br^{\alpha}c_{\alpha_1}^{(\lambda_1)}\cdots c_{\alpha_m}^{(\lambda_m)}A_{\Lambda}^{\alpha}A_{\Lambda}^{* \alpha}\big)D_{\omega_1,T_1}\\
&=\text{SOT}-\lim_{\br\to (1,\dots,1)}\sum_{\alpha\in\Z_+^{m}}\br^{\alpha}c_{(\alpha_1)}^{(\lambda_1)}\cdots c_{\alpha_m}^{(\lambda_m)}T_{\Lambda}^{\alpha} D_{\omega_1,T_1}^2T_{\Lambda}^{* \alpha}\quad (\text{by } ~\eqref{intertwining A})\\
& =\text{SOT}-\lim_{\br\to (1,\dots,1)}\sum_{\alpha\in\Z_+^{m+1}}\br^{\alpha}c^{(1)}_{\alpha_1}c_{\alpha_2}^{(\lambda_1)}\cdots c^{(\lambda_m)}_{\alpha_{m+1}}T_{\Lambda\cup \{1\}}^{\alpha}T_{\Lambda\cup \{1\}}^{* \alpha}\\
&= D_{\clw_{\Lambda\cup\{1\}}, T_{\Lambda\cup\{1\}}}(1,\dots,1).
\end{align*}
For the second part of the lemma we again do a similar computation.
\begin{align*}
Q_{T_1}D_{\clw_{\Lambda},X_{\Lambda}}(1,\dots,1)Q_{T_1} &=
Q_{T_1}\big(\text{SOT}-\lim_{\br\to (1,\dots,1)}\sum_{\alpha\in\Z_+^{m}}\br^{\alpha}c_{\alpha_1}^{(\lambda_1)}\cdots c_{\alpha_m}^{(\lambda_m)}X_{\Lambda}^{\alpha}X_{\Lambda}^{* \alpha}\big)Q_{T_1}\\
&= \text{SOT}-\lim_{\br\to (1,\dots,1)}\sum_{\alpha\in\Z_+^{m}}\br^{\alpha}c_{\alpha_1}^{(\lambda_1)}\cdots c_{\alpha_m}^{(\lambda_m)}T_{\Lambda}^{\alpha}Q_{T_1}^2T_{\Lambda}^{* \alpha}\quad (\text{by \eqref{intertwining X}})\\
& = \text{SOT}-\lim_{\br\to (1,\dots,1)} \sum_{\alpha\in\Z_+^{m}}\br^{\alpha}c_{\alpha_1}^{(\lambda_1)}\cdots c_{\alpha_m}^{(\lambda_m)}T_{\Lambda}^{\alpha}(\text{SOT}-\lim_{k \to \infty}T_1^k T_1^{* k}) T_{\Lambda}^{* \alpha}\\
&= \text{SOT}-\lim_{k \to \infty}T_1^k D_{\clw_{\Lambda}, T_{\Lambda}}(1,\dots,1)T_1^{* k}. 
\end{align*}
In the second last equality, one can interchange of limit and sum using Lebesgue's dominated convergence theorem and the interchange of limits in the last equality can be justified by showing that the double limit exists (see the appendix below for more details). This completes the proof.
\qed

Dilations of pure $\clw$-hypercontractions are very concrete and less complicated to describe compared to that of general $\clw$-hypercontractions. We consider this simpler case first which also helps facilitate the understanding of our dilation method. 
 Recall that for a multi-weight sequence $\clw=(\omega_1,\dots,\omega_n)$ and a non-empty subset $\Lambda=\{\lambda_1,\dots,\lambda_m\}$ of $I$, we denote by $\clw_{\Lambda}$ the multi-weight sequence $(\omega_{\lambda_1},\dots,\omega_{\lambda_m})$. When $\Lambda=\{1,\dots,i\}\subseteq I$ then we simply use $\clw_{i]}$ and $\clw_{[i+1}$ to denote $\clw_{\Lambda}$ and $\clw_{\Lambda^c}$, respectively, with the convention that $\clw_{[n+1}=\emptyset$.

Let $T=(T_1,\ldots,T_n)$ be a pure $\clw$-hypercontraction on $\clh$. Let $M_z$ on $A^2_{\omega_1}(\cld_{\omega_1,T_1})$ be the dilation of $T_1$ with the canonical dilation map $\pi_{\omega_1,T_1} : \clh \to A^2_{\omega_1}(\cld_{\omega_1,T_1})$ as in Theorem~\ref{olo}. 
Then by Proposition~\ref{n-1 tuple lift} and Remark~\ref{pure}, we get an $\clw_{[2}$-hypercontraction $\bm A^{(\bm 2)}=(A^{(2)}_2,\ldots, A^{(2)}_n)$ on $\cld_{\omega_1, T_1}$ such that 
\[ \pi_{\omega_1,T_1}T_i^* = (I_{A^2_{\omega_1}} \otimes A^{(2)}_i)^* \pi_{\omega_1,T_1} \quad \quad (i=2,\ldots,n).\] 
Since each $T_i$ is pure then by the intertwining relation \ref{intertwining A} $\bm A^{(2)}$ is also a pure $\clw_{[2}$-hypercontraction. We now apply Proposition~\ref{n-1 tuple lift} to $\bm A^{(2)}$ as follows. Let $\pi_{\omega_2,A_2^{(2)}} : \cld_{\omega_1,T_1} \to A^2_{\omega_2}(\cld_{\omega_2,A_2^{(2)}})$ be the dilation map of $A^{(2)}_2$. Then we get a pure $\clw_{[3}$-hypercontraction $\bm A^{(3)}= (A^{(3)}_3,\ldots, A^{(3)}_n)$ such that 
\[
\pi_{\omega_2,A^{(2)}_2}A^{(2) *}_i = (I_{A^2_{\omega_2}} \otimes A^{(3) *}_i) \pi_{\omega_2,A^{(2)}_2} \quad \quad (i=3,\ldots,n).
\]
 Set $\Pi_1:= \pi_{\omega_1,T_1}: \clh \to A^2_{\omega_1}(\cld_{\omega_1,T_1})$ and 
\[ \Pi_2:= I_{A^2_{\omega_1}} \otimes \pi_{\omega_2,A^{(2)}_2} : A^2_{\omega_1}(\cld_{\omega_1,T_1}) \to A^2_{\omega_1} \otimes A^2_{\omega_2}(\cld_{\omega_2,A_2^{(2)}}) = 
A^2_{\clw_{2]}}(\cld_{\omega_2,A_2^{(2)}} ). \]
Then a moments thought reveals that the map $\Pi_2\circ\Pi_1$ satisfies
\[ (\Pi_2\circ\Pi_1)T_1^* = M_{z_1}^*(\Pi_2\circ\Pi_1), \quad (\Pi_2\circ\Pi_1)T_2^* = M_{z_2}^*(\Pi_2\circ\Pi_1), \] 
and for all $i = 3, \ldots , n$,
\[ (\Pi_2\circ\Pi_1)T_i^* = (I_{A^2_{\clw_{2]}}(\D^2)} \otimes A^{(3)}_i)^*(\Pi_2\circ\Pi_1). \]
Repeating the above procedure $j$-times we get a pure $\clw_{[j+1}$-hypercontraction $\bm A^{(j+1)}=(A^{(j+1)}_{j+1},\dots,A^{(j+1)}_n )$ 
and an isometry 
\[ \Pi_j := I_{A^2_{\clw_{j-1]}}} \otimes \pi_{\omega_j,A_j^{(j)}} :  A^2_{\clw_{j-1]}}
( \cld_{\omega_{j-1},A_{j-1}^{(j-1)}} )  \to 
A^2_{\clw_{j]}}  ( \cld_{\omega_{j},A_{j}^{(j)}}),\] 
such that for all $i = 1,\ldots, j$, 
\[ (\Pi_j\circ \cdots\circ \Pi_2\circ\Pi_1)T_i^* = M_{z_i}^*(\Pi_j\circ \cdots \circ\Pi_2\circ\Pi_1), \]
and for all $i = j + 1, \ldots , n$, 
\[ (\Pi_j\circ \cdots\circ \Pi_2\circ\Pi_1)T_i^* = (I_{A^2_{\clw_{j]}}} \otimes A^{(j+1)}_i) (\Pi_j\circ \cdots\circ \Pi_2\circ\Pi_1). \]
Thus after $n$-th step we will have the following chain of isometries
\[ 0 \rightarrow \clh \xrightarrow{\Pi_1} A^2_{\omega_1}(\cld_{\omega_1,T_1}) \xrightarrow{\Pi_2} A^2_{\clw_{ 2]}}( \cld_{\omega_2,A_2^{(2)}} ) \xrightarrow{\Pi_3} \cdots \xrightarrow{\Pi_n} A^2_{\clw} ( \cld_{\omega_{n},A_{n}^{(n)}} )\]
such that the isometry $\Pi_T: = \Pi_n \circ \cdots \circ \Pi_2 \circ \Pi_1: \clh \to A^2_{\clw} (\cld_{\omega_{n},A_{n}^{(n)}} )$ satisfies
\[ \Pi_T T_i^* = M_{z_i}^*\Pi_T \quad (i=1,\dots,n).\]
Thus $T$ dilates to the weighted Bergman shift $(M_{z_1}, \ldots , M_{z_n})$ on $A^2_{\clw} (\cld_{\omega_{n},A_{n}^{(n)}} )$ via the dilation map $\Pi_T$.
We summarize this in the next result.

\begin{Theorem}\label{model-gen-pure}
Let $\clw$ be a multi-weight sequence and let $ T=(T_1,\ldots ,T_n)$ be a pure $\clw$-hypercontraction on $\clh$. Then there exist a Hilbert space $\cle$ and 
a joint $(M_{z_1}^*,\dots, M_{z_n}^*)$-invariant subspace $\clq$ of $ A^2_{\clw} (\cle )$ such that  
\[ (T_1,\ldots,T_n) \cong P_{\clq}(M_{z_1}, \ldots, M_{z_n})|_{\clq}. \] 
\end{Theorem}


We now consider dilations of general $\clw$-hypercontractions and find their explicit dilation. In order to make the proof of the dilation result shorter, we take out some part of the proof and prove it as a separate lemma. The lemma is a special case of one operator being a co-isometry.
\begin{Lemma}\label{model lemma}
Let $(V, Y_1,\ldots,Y_{n-1})\in \clb(\clh)^n$ be an $n$-tuple of commuting contractions such that $V$ is a co-isometry. Suppose that the $(n-1)$-tuple $Y^\prime =(Y_1, \ldots, Y_{n-1})$ is an $\clw$-hypercontractive tuple and dilates to an $\clw$-hypercontractive tuple $R^\prime=(R_1, \ldots, R_{n-1})$ on $\clk$ through the dilation map $\Pi:\clh\to \clk:= \bigoplus_{\Lambda\subseteq I \setminus\{n\} } A^2_{\clw_{\Lambda}}(\cle_{\Lambda})$, where 
\[ 
\Pi= \bigoplus_{\Lambda\subseteq I \setminus \{n\} } \Pi_{\Lambda}
\quad \text{and} \quad 
R_i= \left(\bigoplus_{\Lambda\subseteq I \setminus \{n\}} R_i^{\Lambda}\right), 
\]
such that for each $\Lambda = \{\lambda_1,\ldots, \lambda_m \}$, 
$\Pi_{\Lambda}:\clh \to A^2_{\clw_{\Lambda}}(\cle_{\Lambda})$ defined by
\[ 
\Pi_{\Lambda}h(\z) =  \sum_{\alpha \in \Z_+^m} \Big( \frac{1}{\omega_{\alpha_1}^{(\lambda_1)} \cdots \omega_{\alpha_m}^{(\lambda_m)}} \Delta_{\Lambda} Y_{\Lambda}^{*\alpha} h \Big)\z^{\alpha},
\] $R^{\Lambda}_{\lambda_j}=M_{z_j}$ on $A^2_{\clw_{\Lambda}}(\cle_\Lambda)$ $(j=1,\dots,m)$ and for $i\notin\Lambda$,  $R_{i}^{\Lambda}=I_{A^2_{\clw_{\Lambda}}}\otimes V^{\Lambda}_i$ for some co-isometry $V^{\Lambda}_i$ on $\cle_{\Lambda}= \overline{ran}\Delta_{\Lambda}$, with 
$\Delta_{\Lambda}\in B(\clh)$, satisfying  $\Delta_{\Lambda} Y_i^*=(V_{i}^{\Lambda})^*\Delta_{\Lambda}$ for all $i\notin \Lambda$. 
If the co-isometry $V$ satisfies the relations 
\[V \Delta^{*}_{\Lambda}\Delta_{\Lambda} V^*= \Delta^{*}_{\Lambda}\Delta_{\Lambda} \quad \text{for each}\quad \Lambda \subseteq \{1,\ldots,n-1\},\]
then $V$ lifts to a co-isometry $W$ on $\clk$ such that the tuple $(V, Y_1, \ldots, Y_{n-1})$ on $\clh$ dilates to 
$R=(W, R_1, \ldots, R_{n-1})$ on $\clk$. 
\end{Lemma}

\textit{Proof:}
By the hypothesis, for each $\Lambda \subseteq \{1,\ldots, n-1 \}$, $V \Delta^{*}_{\Lambda}\Delta_{\Lambda} V^* = \Delta^{*}_{\Lambda}\Delta_{\Lambda}$. Therefore for each $ \Lambda \subseteq \{1,\ldots, n-1\}$, by the Douglas factorization lemma, there exists a co-isometry 
$W_{\Lambda}$ on $\cle_{\Lambda}$ such that 
\[
W_{\Lambda}^* \Delta_{\Lambda}=\Delta_{\Lambda} V^*.
\]
We define a co-isometry $W$ on  $\clk=\bigoplus_{\Lambda \subseteq I\setminus \{n\}} A^2_{\clw_{\Lambda}}(\cle_{\Lambda}) $ as
\[
W =\left(\bigoplus_{\Lambda \subseteq I \setminus \{n\}}  (I_{A^2_{\clw_{\Lambda} } } \otimes W_{\Lambda})\right). 
\] 
Now, we observe that for any $\Lambda \subseteq \{1,\ldots,n-1\}$, $h \in \cle_{\Lambda}$ and $\z \in \D^{|\Lambda|}$,
\begin{align*}
\Pi_{\Lambda}V^*h(\z) &= \sum_{\alpha \in \Z_+^m} \Big( \frac{1}{\omega_{\alpha_1}^{(\lambda_1)} \cdots \omega_{\alpha_m}^{(\lambda_m)}} \Delta_{\Lambda} Y_{\Lambda}^{*\alpha} V^*h \Big) z_{\lambda_1}^{\alpha_1} \cdots z_{\lambda_m}^{\alpha_m} \\
&= \sum_{\alpha \in \Z_+^m} \Big( \frac{1}{\omega_{\alpha_1}^{(\lambda_1)} \cdots \omega_{\alpha_m}^{(\lambda_m)}} \Delta_{\Lambda} V^* Y_{\Lambda}^{*\alpha}h \Big) z_{\lambda_1}^{\alpha_1} \cdots z_{\lambda_m}^{\alpha_m} \\
&= \sum_{\alpha \in \Z_+^m} \Big( \frac{1}{\omega_{\alpha_1}^{(\lambda_1)} \cdots \omega_{\alpha_m}^{(\lambda_m)}} W_{\Lambda}^*\Delta_{\Lambda} Y_{\Lambda}^{*\alpha} h \Big) z_{\lambda_1}^{\alpha_1} \cdots z_{\lambda_m}^{\alpha_m}\\
&= W_{\Lambda}^* \Pi_{\Lambda} h(\z) .
\end{align*}
Thus $\Pi V^* = W^* \Pi$. For the commutativity of the tuple $(W, R_1, \ldots, R_{n-1})$, we fix $i\in\{1,\dots,n-1\}$ and show that $WR_i=R_iW$. To this end, it is enough to show that for each $\Lambda \subseteq \{1,\ldots, n-1\}$, $(I_{A^2_{\clw_{\Lambda} } } \otimes W_{\Lambda})R_i^{\Lambda}=R_i^{\Lambda}(I_{A^2_{\clw_{\Lambda} } } \otimes W_{\Lambda})$. If $i=\lambda_j\in\Lambda$, then $R_i^{\Lambda}=M_{z_j}$ and there is nothing to prove. If $i\notin \Lambda$, then $R_i^{\Lambda}=I_{A^2_{\clw_{\Lambda} } }\otimes V_i^{\Lambda}$ and the commutativity can be read from the following:
\begin{align*}
    W_{\Lambda}^* V^{\Lambda*}_i \Delta_{\Lambda}
    &= W_{\Lambda}^* \Delta_{\Lambda} Y_i^*\\
    &=\Delta_{\Lambda} V^* Y^*_i\\
    &=\Delta_{\Lambda} Y^*_i V^*\\
    &= V^{\Lambda*}_i W_{\Lambda}^* \Delta_{\Lambda} .
\end{align*}
This completes the proof. \qed

We are now ready to prove the model for general $\clw$-hypercontractive tuples. It is an exact generalization of the model given in Theorem 2.8 in \cite{CV}.
\begin{Theorem} \label{model-gen-tuple}
Let $T=(T_1,\ldots,T_n)$ be an $\clw$-hypercontraction on $\clh$ for some multi-weight sequence $\clw$.
Then there exist $\Delta_{\Lambda}\in B(\clh)$ corresponding to each subset $\Lambda$ of $I$ with $\cle_{\Lambda}=\overline{ran}\Delta_{\Lambda}$, an isometry 
\[
\Pi:\clh\to \clk:= \bigoplus_{\Lambda\subseteq I}A^2_{\clw_{\Lambda}}(\cle_{\Lambda}), 
\]
and an $n$-tuple of commuting contractions $R=(R_1,\ldots,R_n)$ on $\clk$ such that 
\[
\Pi T_i^*=R_i^*\Pi\quad (i=1,\dots,n),
\]
where with respect to the above decomposition of $\clk$
 \begin{equation}\label{explicit form}
 \Pi= \bigoplus_{\Lambda\subseteq I}\Pi_{\Lambda}\ \text{and }
 R_i= \left(\bigoplus_{\Lambda\subseteq I}R_i^{\Lambda}\right)
 \end{equation}
such that if $\Lambda=\{\lambda_1,\dots,\lambda_m\}$ then $\Pi_{\Lambda}: \clh\to A^2_{\clw_{\Lambda}}(\cle_{\Lambda})$ is defined by 
\[ \Pi_{\Lambda}h(\z) =  \sum_{\alpha \in \Z_+^m} \Big( \frac{1}{\omega_{\alpha_1}^{(\lambda_1)} \cdots \omega_{\alpha_m}^{(\lambda_m)}} \Delta_{\Lambda} T_{\Lambda}^{*\alpha} h \Big)\z^{\alpha}
\]
and 
\[
R_i^{\Lambda} =\left\{\begin{array}{rl}
I_{A^2_{\clw_{\Lambda}}}\otimes V^{\Lambda}_i & \text{ if } i\notin \Lambda\\
M_{z_j} & \text{ if } i=\lambda_j\in\Lambda
\end{array}
\right.
\]
for some co-isometry $V^{\Lambda}_i$ on $\cle_{\Lambda}$. 
Moreover, for every subset $\Lambda=\{\lambda_1,\dots,\lambda_m\}$ of $I$, 
\begin{align*}\Delta_{\Lambda}^{*} \Delta_{\Lambda}& = 
\text{SOT}-\lim_{\beta \to \infty} T_{\Lambda^{c}}^{ \beta } D_{\clw_{\Lambda},T_{\Lambda}}(1,\dots,1) T_{\Lambda^{c}}^{* \beta}
\end{align*}
and for all $i\notin \Lambda$, $\Delta_{\Lambda} T_i^*=(V_{i}^{\Lambda})^*\Delta_{\Lambda}$. 
\end{Theorem}
\textit{Proof.} We prove the theorem by induction on $n$. For $n=1$, if $T$ is a $\omega$-hypercontraction on $\clh$ then by Theorem~\ref{olo} the result holds with $\cle_{\{1\}}=\cld_{\omega,T}$, $\cle_{\emptyset}=\clq_T$,
$R^{\{1\}}_1=M_z$ on $A^2_{\omega}(\cld_{\omega,T})$, $R^{\emptyset}_1=U$, $\Delta_{\{1\}}=D_{\omega, T}$, $\Delta_{\emptyset}= Q_T$ 
and $\Pi= \Pi_{\{1\}} \oplus \Pi_{\{\emptyset\}}$ with 
\[
\Pi_{\{1\}}h(z)= \sum_{k \geq 0} \Big( \frac{1}{\omega_k} \Delta_{\{1\}} T^{*k}h \Big) z^k\ \text{and } \Pi_{\{\emptyset \}}h = Q_Th \quad (h\in \clh, z\in \D).
\]

Now we assume that the result is true for $n-1$. Let $T=(T_1,\dots,T_n)$ be an $\clw$-hypercontraction.
Suppose $(M_{z}\oplus U)$ on $A^2_{\omega_1}(\cld_{\omega_1,T_1})\oplus \clq_{T_1}$ is the dilation of $T_1$ with the canonical dilation map $\Pi_{\omega_1, T_1}:\clh\to A^2_{\omega_1}(\cld_{\omega_1,T_1})\oplus \clq_{T_1} $ as obtained in Theorem~\ref{olo}. Then applying Proposition~\ref{n-1 tuple lift} we get $\clw_{[2}$-hypercontractions $A=(A_2,\dots, A_{n})$ on $\cld_{\omega_1,T_1}$ and $X=(X_2,\dots, X_n)$ on $\clq_{T_1}$
such that 
\[
\Pi_{\omega_1,T_1}T_i^*=V_i^*\Pi_{\omega_1, T_1}\quad (i=2,\dots,n),
\]
where $V_i=(I_{A^2_{\omega_1}}\otimes A_i)\oplus X_i$ for all $i=2,\dots,n$. We now apply the induction hypothesis to both $A$ and $X$. 

Since $A$ is an $\clw_{[2}$-hypercontraction, then by the hypothesis we get Hilbert spaces $\cle_{\Lambda}^{'}=\overline{ran}\Delta^{'}_{\Lambda}$ for all subset $\Lambda$ of $\{2,\dots,n\}$, an isometry 
\[\Pi_A: \cld_{\omega_1,T_1}\to \clk':= \bigoplus_{\Lambda\subseteq \{2,\dots,n\}}A^2_{(\clw_{[2})_{\Lambda}}(\cle_{\Lambda}^{'})=
\bigoplus_{\Lambda\subseteq \{2,\dots,n\}}A^2_{\clw_{\Lambda}}(\cle_{\Lambda}^{'}),
\]
where $\Pi_A = \oplus_{\Lambda \subseteq \{2,\ldots,n\}} \Pi^A_{\Lambda} $ and for $\Lambda= \{\lambda_2,\ldots, \lambda_m\}$, $\Pi^A_{\Lambda}:\cld_{\omega_1,T_1}\to A^2_{\clw_{\Lambda}}(\cle_{\Lambda}^{'}) $
is defined by 
\[ \Pi^A_{\Lambda}h(\z) =  \sum_{\alpha \in \Z_+^{m-1}} 
\Big( \frac{1}{\omega_{\alpha_2}^{(\lambda_2)} \cdots \omega_{\alpha_m}^{(\lambda_m)}} \Delta^{'}_{\Lambda} A_{\Lambda}^{*\alpha} h \Big)\z^{\alpha}
\] 
and $\Delta^{'*}_{\Lambda} \Delta^{'}_{\Lambda}= \text{SOT}-\lim_{\beta \to \infty} 
A_{\Lambda^c}^{\beta } D_{\clw_{\Lambda}, A_{\Lambda}}(1,\ldots,1) A_{\Lambda^c}^{* \beta } $ $(\beta\in \Z_+^{|\Lambda^c|})$. 
We also have an $(n-1)$-tuple of commuting contractions $R'=(R_2',\dots, R_n')$ on $\clk$ having the structure as in ~\eqref{explicit form} such that $\Pi_A A_i^*=R_i^{' *}\Pi_A$. 

Once again applying the hypothesis to the $\clw_{[2}$-hypercontraction $X$ on $\clq_{T_1}$ we also get Hilbert spaces $\cle_{\Lambda}^{''}=\overline{ran}\Delta^{''}_{\Lambda}$ for all subset $\Lambda$ of $\{2,\dots,n\}$, an isometry \[\Pi_X: \clq_{T_1}\to \clk^{''}:= \bigoplus_{\Lambda\subseteq \{2,\dots,n\}}A^2_{(\clw_{[2})_{\Lambda}}(\cle_{\Lambda}^{''})
=\bigoplus_{\Lambda\subseteq \{2,\dots,n\}}A^2_{\clw_{\Lambda}}(\cle_{\Lambda}^{''}),
\]
 where $\Pi_X = \oplus_{\Lambda \subseteq \{2,\ldots,n\}} \Pi^X_{\Lambda} $ and for $\Lambda= \{\lambda_2,\ldots, \lambda_m\}$, $\Pi^X_{\Lambda}:\clq_{T_1}\to A^2_{\clw_{\Lambda}}(\cle_{\Lambda}^{''})$ is defined by 
\[ \Pi^X_{\Lambda}h(\z) =  \sum_{\alpha \in \Z_+^{m-1}} 
\Big( \frac{1}{\omega_{\alpha_2}^{(\lambda_2)} \cdots \omega_{\alpha_m}^{(\lambda_m)}} \Delta^{''}_{\Lambda} X_{\Lambda}^{*\alpha} h \Big)\z^{\alpha}
\] 
and $\Delta^{''*}_{\Lambda} \Delta^{''}_{\Lambda}= \text{SOT}-\lim_{\beta \to \infty} 
X_{\Lambda^c}^{\beta } D_{\clw_{\Lambda}, X_{\Lambda}}(1,\ldots,1) X_{\Lambda^c}^{* \beta } $ $(\beta\in\Z_+^{|\Lambda^c|})$.
We also have an $(n-1)$-tuple of commuting contractions $R^{''}=(R_2^{''},\dots, R_n^{''})$ on $\clk^{''}$ having structure as in ~\eqref{explicit form} such that $\Pi_X X_i^*=R_i^{'' *}\Pi_X$.
Since $U$ is a co-isometry on $\clq_{T_1}$ which commutes with $X$, then by Lemma~\ref{model lemma} we get a co-isometry $W$ on $\clk^{''}$ such that $W$ commutes with $R^{''}$ and $\Pi_X U^*=W^*\Pi_X$. 
We now have all the ingredient to construct the required dilation of $T$. For a subset $\Lambda$ of $I$, we set 
\[
\cle_{\Lambda}:=\left\{\begin{array}{cc}
\cle'_{\Lambda\setminus\{1\}} & \text{if } 1\in\Lambda\\
\cle^{''}_{\Lambda}& \text{otherwise}
\end{array}\right.\ \text{ and } \clk:=\bigoplus_{\Lambda\subseteq I} A^2_{\clw_{\Lambda}}(\cle_{\Lambda}).
\]
Then note that  
\[
A^2_{\omega_1}\otimes\clk'=\bigoplus_{\Lambda\subseteq \{2,\dots,n\}}A^2_{\omega_1}\otimes A^2_{\clw_{\Lambda}}(\cle_{\Lambda}^{'})=\bigoplus_{\Lambda\subseteq I, 1\in\Lambda} A^2_{\clw_{\Lambda}}(\cle_{\Lambda})\ \text{ and }
\clk=(A^2_{\omega_1}\otimes \clk')\oplus \clk^{''}.
\]
We consider the dilation map for $T$ as 
\[
\Pi=((I_{A^2_{\omega_1}}\otimes \Pi_{A})\oplus \Pi_{X})  \circ\Pi_{\omega_1,T_1}: \clh\to \bigoplus_{\Lambda\subseteq \{2,\dots,n\}}A^2_{\omega_1}\otimes A^2_{\clw_{\Lambda}}(\cle_{\Lambda}^{'})\bigoplus_{\Lambda\subseteq \{2,\dots,n\}}A^2_{\clw_{\Lambda}}(\cle_{\Lambda}^{''})=\clk. 
\]
More explicitly, the dilation map has the following decomposition 
\begin{align*}
\Pi &= \bigoplus_{\Lambda \subseteq \{2,\dots,n\}} \big((I_{A^2_{\omega_1}}\otimes \Pi^A_{\Lambda}) \circ \pi_{\omega_1,T_1} \big) \bigoplus_{\Lambda \subseteq \{2,\dots,n\}} \big(\Pi^X_{\Lambda} \circ Q_{T_1} \big)\\
&= \bigoplus_{\Lambda\subseteq I} \Pi_{\Lambda},
\end{align*}
where  
\[
\Pi_{\Lambda}:=\left\{\begin{array}{cc}
(I_{A^2_{\omega_1}}\otimes \Pi^A_{\Lambda\setminus \{1\}}) \circ \pi_{\omega_1,T_1}   & \text{if } 1\in\Lambda\\
\Pi^X_{\Lambda} \circ Q_{T_1} & \text{otherwise.}\
\end{array}\right.
\]
Moreover, for $\Lambda\subseteq I$ with $1\in\Lambda$ and $h\in\clh$,
\begin{equation}\label{1 in lambda}
\Pi_{\Lambda} h (\z) = \sum_{\alpha \in \Z_+^m} 
\frac{1}{\omega^{(1)}_{\alpha_1}} \cdots \frac{1}{\omega^{(m)}_{\alpha_m}} \big(\Delta^{'}_{\Lambda\setminus\{1\}} D_{\omega_1,T_1}T^{* \alpha}h \big) \z^{\alpha},
\end{equation}
and for $1\notin \Lambda$, 
\begin{equation}\label{1 notin lambda}
\Pi_{\Lambda}h (\z) = \sum_{\alpha \in \Z_+^{m-1}} 
\frac{1}{\omega^{(2)}_{\alpha_1}} \cdots \frac{1}{\omega^{(m)}_{\alpha_{(m-1)}}} \big(\Delta^{''}_{\Lambda} Q_{T_1}T^{* \alpha}h \big) \z^{\alpha}.
\end{equation}
Here for the above description of the dilation map we have used the intertwining relations $A_i^*D_{\omega_1,T_1=}D_{\omega_1,T_1}T_i^*$ and $X_i^*\clq_{T_1}=\clq_{T_1}T_i^*$ for all $i=2,\dots,n$, which can be read from equations ~\eqref{intertwining A} and ~\eqref{intertwining X}, respectively.

The dilating tuple of commuting contractions $R=(R_1,\dots, R_n)$ on $\clk$ is defined with respect to the decomposition $\clk=(A^2_{\omega_1}\otimes\clk')\oplus \clk^{''}$ as
\[
R_1=(M_{z}\otimes I_{\clk'})\oplus W \ \text{ and }
R_i=(I_{A^2_{\omega_1}}\otimes R'_i)\oplus R_i^{''}\quad (i=2,\dots,n).
\]
We now do a routine calculation to see that $R$ is a dilation of $T$. First note that 
\begin{align*}
\Pi T_1^* &= ((I_{A^2_{\omega_1}}\otimes \Pi_{A})\oplus \Pi_{X})  (M_{z}\oplus U)^*\Pi_{\omega_1,T_1}\\
&= ((I_{A^2_{\omega_1}}\otimes \Pi_{A}) M_{z}^*\oplus \Pi_{X}U^*)\Pi_{\omega_1,T_1}\\
&=((M_{z}^*\otimes I_{\clk'})\oplus W^*)((I_{A^2_{\omega_1}}\otimes \Pi_{A})\oplus \Pi_{X})\Pi_{\omega_1,T_1}\\
&=R_1^*\Pi,
\end{align*}
and similarly for all $i=2,\dots,n$,
\begin{align*}
\Pi T_i^* &=((I_{A^2_{\omega_1}}\otimes \Pi_{A})\oplus \Pi_{X})  (I_{A^2_{\omega_1}}\otimes A_i\oplus X_i)^*\Pi_{\omega_1,T_1}\\
&=((I_{A^2_{\omega_1}}\otimes R_i')\oplus R_i^{''})^*((I_{A^2_{\omega_1}}\otimes \Pi_{A})\oplus \Pi_{X})  \Pi_{\omega_1,T_1}\\
&= R_i^*\Pi.
\end{align*}
We left it to the reader to check that each of the operator $R_i$ has the form as in ~\eqref{explicit form}. To complete the proof, we now need to construct $\Delta_{\Lambda}\in \clb(\clh)$ corresponding to each subset $\Lambda$ of $I$ which satisfies the moreover part of the theorem. To this end, we define 
\[
\Delta_{\Lambda}:=\left\{\begin{array}{cc}
 \Delta^{'}_{\Lambda\setminus \{1\}} D_{\omega_1,T_1}  & \text{if } 1\in\Lambda\\
 \Delta^{''}_{\Lambda} Q_{T_1} & \text{otherwise.}
\end{array}\right.
\]
Rest of the proof is divided into two cases.

\textbf{Case I:} $1\in\Lambda$. In this case, setting $\Gamma:=\Lambda\setminus\{1\}$, we have
\begin{align*}
\Delta_{\Lambda}^*\Delta_{\Lambda}& = D_{\omega_1,T_1}\Delta^{'*}_{\Gamma}\Delta^{'}_{\Gamma} D_{\omega_1,T_1}\\
& =\text{SOT}-\lim_{\beta \to \infty} D_{\omega_1,T_1}A_{\Lambda^c}^{\beta} D_{\clw_{\Gamma},A_{\Gamma}}(1,\dots,1)A_{\Lambda^c}^{*\beta}D_{\omega_1,T_1}\quad  (\beta\in \Z_+^{|\Lambda^c|})\\
& =\text{SOT}-\lim_{\beta \to \infty}
T_{\Lambda^c}^{\beta}D_{\omega_1,T_1}D_{\clw_{\Gamma},A_{\Gamma}}(1,\dots,1)D_{\omega_1,T_1}T_{\Lambda^c}^{* \beta}\\
& =\text{SOT}-\lim_{\beta \to \infty}
T_{\Lambda^c}^{\beta} D_{\clw_{\Lambda}, T_{\Lambda}}(1,\dots,1)T_{\Lambda^c}^{* \beta}\quad (\text{ by part (i) of Lemma~\ref{useful lemma}}).
\end{align*}
Also if $i\notin \Lambda$ then by construction $\Delta_{\Gamma}'A_i^*= (V_i^{' \Gamma})^*\Delta_{\Gamma}' $, where $V_i^{'\Gamma}$ is the co-isometry corresponding to $R_{i}^{'\Gamma}$, that is,  $R_{i}^{'\Gamma}=I_{A^2_{\clw_{\Gamma}}}\otimes V_i^{'\Gamma}$.
Also since $R_i^{\Lambda}=I_{A^2_{\omega_1}}\otimes R_{i}^{'\Gamma}=I_{A^2_{\clw_{\Lambda}}}\otimes V_i^{'\Gamma}$, the co-isometries corresponding to $R_i^{\Lambda}$ and $R_i^{'\Gamma}$ are the same, that is, $V_i^{\Lambda}=V_i^{'\Gamma}$. Consequently,
\begin{align*}
\Delta_{\Lambda}T_i^*&=\Delta_{\Gamma}'D_{\omega_1, T_1}T_i^*
=\Delta_{\Gamma}'A_i^*D_{\omega_1,T_1}
= (V_i^{'\Gamma})^*\Delta_{\Gamma}'D_{\omega_1,T_1}
=(V_i^{\Lambda})^*\Delta_{\Lambda}.
\end{align*}

\textbf{Case II:} $1 \notin \Lambda$. In this case, setting $\Gamma=\Lambda\cup \{1\}$, we have 
\begin{align*}
\Delta_{\Lambda}^*\Delta_{\Lambda}& = Q_{T_1}\Delta_{\Lambda}^{'' *}\Delta_{\Lambda}^{''}Q_{T_1}\\
&= \text{SOT}-\lim_{\beta \to \infty}
Q_{T_1}X_{\Gamma^c}^{\beta}D_{\clw_{\Lambda}, X_{\Lambda}}(1,\dots,1)X_{\Gamma^c}^{* \beta} Q_{T_1}\quad (\beta\in \Z_+^{|\Gamma^c|})\\
&= \text{SOT}-\lim_{\beta \to \infty}
T_{\Gamma^c}^{\beta}Q_{T_1}D_{\clw_{\Lambda}, X_{\Lambda}}(1,\dots,1)Q_{T_1}T_{\Gamma^c}^{* \beta}\\
&= \text{SOT}-\lim_{\beta \to \infty}T_{\Lambda^c}^{\beta} D_{\clw_{\Lambda}, T_{\Lambda}}(1,\dots,1)T_{\Lambda^c}^{* \beta}\quad (\text{by part (ii) of Lemma~\ref{useful lemma}}).
\end{align*}
If $i\notin \Lambda$ and $i\ge 2$ then by the hypothesis $\Delta^{''}_{\Lambda}X_i^*=(V^{'' \Lambda}_{i})^*\Delta^{''}_{\Lambda}$, where $V^{'' \Lambda}_{i}$ is the co-isometry corresponding to $R^{'' \Lambda}_{i}$. In this case, by construction the co-isometry corresponding to $R_i^{\Lambda}$ and $R^{'' \Lambda}_{i}$ are the same, that is, $V^{\Lambda}_i=V^{'' \Lambda}_i$. Consequently, 
\[
\Delta_{\Lambda}T_i^*=\Delta^{''}_{\Lambda}Q_{T_1}T_i^*=\Delta^{''}_{\Lambda}X_i^*Q_{T_1}= V^{'' \Lambda}_{i}\Delta^{''}_{\Lambda}Q_{T_1}=V^{\Lambda}_{i}\Delta_{\Lambda}.
\]
Finally, for $i=1$ the co-isometry corresponding to $R_1^{\Lambda}$ is $W^{\Lambda}_1$ and it follows from Lemma~\ref{model lemma}
that $\Delta^{''}_{ \Lambda}U^*=W^{\Lambda *}_1\Delta^{'' \Lambda} $. Thus 
\[\Delta_{\Lambda} T_1^*=\Delta^{''}_{ \Lambda}Q_{T_1}T_1^*=\Delta^{''}_{ \Lambda}U^* Q_{T_1}=W^{\Lambda *}_1\Delta^{''}_{ \Lambda} Q_{T_1}=W^{\Lambda *}_1 \Delta_{\Lambda}.
\]
This completes the proof.
\qed

\newsection{Appendix}
Let $T$ be an $\clw$-hypercontraction corresponding to a multi-weight sequence $\clw$. Then for a subset $\Lambda=\{\lambda_1,\dots,\lambda_m\}$ of $I$,
consider the map $f:(0,1)^m \times \Z^{n-m}_+ \to \clb(\clh)$, defined by
\begin{align*}
    f(\bm r^\prime, \beta)=T_{\Lambda^{c}}^{ \beta } D_{\clw_{\Lambda},T_{\Lambda}}(\bm r^\prime) T_{\Lambda^{c}}^{* \beta},
\end{align*}
where $\bm r^\prime=(r_1, \ldots, r_m)\in(0,1)^m$ and $\beta=(\beta_1, \ldots, \beta_{n-m})\in \Z_+^{n-m}$. We claim that $\text{SOT}-\lim_{(\bm r^\prime, \beta)\to (\mathbf e^\prime, \infty)} f(\bm r^\prime, \beta)$ exists, where $\mathbf e^\prime=(1,\ldots,1)$. Indeed, it is enough to show that for $\bm r^\prime\leq \bm s^\prime$ and $\alpha\leq \beta$, $f(\bm s^{\prime},\beta)\le f(\bm r^\prime, \alpha)$.
Let $j\in \Lambda^{c}$ and without any loss of generality assume that $\lambda_m<i$. Then the multi-weight sequence $\clw^\prime=(\clw_\Lambda, \mathds 1)\in S(\clw_{\Lambda^\prime})$, where $\Lambda^\prime=\{\lambda_1, \ldots, \lambda_m, i\}$. Since $T_{\Lambda^\prime}$ is a $\clw_{\Lambda^\prime}$-hypercontraction, then for any $r\in (0,1)$
    \begin{align}\label{Inequ_1}
        D_{\clw^\prime,T_{\Lambda^\prime}}(\bm r^\prime, r)= D_{\clw_{\Lambda},T_{\Lambda}}(\bm r^\prime)-r T_i D_{\clw_{\Lambda},T_{\Lambda}}(\bm r^\prime) T^*_i\geq 0\quad (r\in (0,1)).
    \end{align}
By Remark \ref{split limit}, taking limit as $r\to 1$ in \ref{Inequ_1} we get 
\begin{align*}
    T_i D_{\clw_{\Lambda},T_{\Lambda}}(\bm r^\prime) T^*_i\leq D_{\clw_{\Lambda},T_{\Lambda}}(\bm r^\prime)
\end{align*}
for all $i\in \Lambda^c$. Consequently, for any $\beta=(\beta_1, \ldots, \beta_{n-m})\in \Z_+^{n-m}$,
\begin{align*}
     T_{\Lambda^{c}}^{ \beta } D_{\clw_{\Lambda},T_{\Lambda}}(\bm r^\prime) T_{\Lambda^{c}}^{ * \beta }\leq D_{\clw_{\Lambda},T_{\Lambda}}(\bm r^\prime). 
\end{align*}
Now, for $\alpha\leq \beta$, 
\begin{align*}
     T_{\Lambda^{c}}^{ \alpha} D_{\clw_{\Lambda},T_{\Lambda}}(\bm r^\prime) T_{\Lambda^{c}}^{ * \alpha }- T_{\Lambda^{c}}^{ \beta } D_{\clw_{\Lambda},T_{\Lambda}}(\bm r^\prime) T_{\Lambda^{c}}^{ * \beta }= T_{\Lambda^{c}}^{ \alpha } \Big( D_{\clw_{\Lambda},T_{\Lambda}(\bm r^\prime) - T_{\Lambda^{c}}^{ \beta-\alpha } D_{\clw_{\Lambda},T_{\Lambda}}(\bm r^\prime) T_{\Lambda^{c}}^{* \beta-\alpha }}\Big) T_{\Lambda^{c}}^{ * \alpha}\ge 0.
\end{align*}
That is, 
 \begin{align}\label{Inequ_2}
     T_{\Lambda^{c}}^{ \beta } D_{\clw_{\Lambda},T_{\Lambda}}(\bm r^\prime) T_{\Lambda^{c}}^{ * \beta }\leq  T_{\Lambda^{c}}^{ \alpha } D_{\clw_{\Lambda},T_{\Lambda}}(\bm r^\prime) T_{\Lambda^{c}}^{ * \alpha }
 \end{align}
 for all $\alpha\leq \beta$.
Finally, using Proposition \ref{Pro_Decr} and the inequality \ref{Inequ_2}, for $\bm r^\prime\leq \bm s^\prime$ and $\alpha\leq \beta$,  
\begin{align*}
      f(\bm r^\prime, \alpha)=T_{\Lambda^{c}}^{ \alpha } D_{\clw_{\Lambda},T_{\Lambda}}(\bm r^\prime) T_{\Lambda^{c}}^{ * \alpha }\geq T_{\Lambda^{c}}^{ \alpha } D_{\clw_{\Lambda},T_{\Lambda}}(\bm s^\prime) T_{\Lambda^{c}}^{ * \alpha }\geq  T_{\Lambda^{c}}^{ \beta } D_{\clw_{\Lambda},T_{\Lambda}}(\bm s^\prime) T_{\Lambda^{c}}^{ * \beta }= f(\bm s^\prime, \beta). 
\end{align*}
This proves the claim.

\vspace{0.1in} \noindent\textbf{Acknowledgement:}
The authors are grateful to Prof. Jaydeb Sarkar for some insightful discussions at the beginning of this project. The first named author acknowledges IIT Bombay for warm hospitality, and his research is supported by the DST-INSPIRE Faculty Fellowship No. DST/INSPIRE/04/2020/001250. The second author is supported by the Mathematical Research Impact Centric Support (MATRICS) grant, File No: MTR/2021/000560, by the Science and Engineering Research Board (SERB), Department of Science \& Technology (DST), Government of India. The third named author acknowledges IIT Bombay for its warm hospitality. The research of the third named author is supported by the institute post-doctoral fellowship of IIT Bombay.



\bibliographystyle{amsplain}

\end{document}